\documentclass[11pt]{article}
\usepackage[margin=2cm, centering]{geometry}
\usepackage{amsmath,amssymb,mathtools} % maths symbols
\usepackage{mathrsfs} % for \mathscr
\usepackage{amsthm} % for the proof environment
\usepackage{hyperref} % for links to references and equations
\usepackage{xcolor} % for comments, reminders

\newtheorem{theorem}{Theorem}[section]

\newtheorem{lemma}[theorem]{Lemma}
\newtheorem{proposition}[theorem]{Proposition}

\newtheorem{remark}[theorem]{Remark}
\newtheorem{assumption}[theorem]{Assumption}

\newtheorem*{claim}{Claim}

\begin{document}
    \title{On gradient stability in nonlinear PDE models \\ and inference in interacting particle systems}
    \date{}
    \author{Aur\'elien Castre and Richard Nickl\\ \\ University of Cambridge \\ ~ \\ \today}
    \maketitle

\begin{abstract}
We consider general parameter to solution maps $\theta \mapsto \mathcal G(\theta)$ of non-linear partial differential equations and describe an approach based on a Banach space version of the implicit function theorem to verify the gradient stability condition of \cite{NicklWang2024_polynomialTimeLangevin} for the underlying non-linear inverse problem, providing also injectivity estimates and corresponding statistical identifiability results. We illustrate our methods in two examples involving a non-linear reaction diffusion system as well as a McKean--Vlasov interacting particle model, both with periodic boundary conditions.  We apply our results to prove the polynomial time convergence of a Langevin-type algorithm sampling the posterior measure of the interaction potential arising from a discrete aggregate measurement of the interacting particle system.

\end{abstract}

\setcounter{tocdepth}{3}
\tableofcontents
%\break

\section{Introduction}

Many algorithms used in contemporary computational mathematics employ the gradient $\nabla H$ of a map $H: \mathbb R^D \to \mathbb R$ defined on the Euclidean space $\mathbb R^D$. This includes the tasks of minimising $H$ over its potentially high-dimensional ($D \gg 1$) domain by (stochastic) gradient descent, or of generating random samples drawn from the Gibbs-type probability measure
\begin{equation*}
    \mathop{}\!\mathrm{d} \mu(\theta) \propto e^{-H(\theta)} \mathop{}\!\mathrm{d}\theta, \quad \theta \in \mathbb{R}^{D},
\end{equation*}
by Langevin-type methods. In the sampling problem, implementations often rely on gradient based Markov Chain Monte Carlo (MCMC) such as the Metropolis-adjusted Langevin algorithm (MALA) \cite{Besag1994,RobertsTweedie1996,RobertsRosenthal1998}. These methods have proved popular in a wide array of applications -- too many to be listed here, but let us mention non-linear PDE type inverse problems and data assimilation  \cite{Stuart2010, CRSW13, RC15, GW25} as an instance that will be studied in the present article. Theoretical convergence guarantees for such algorithms have been extensively studied in recent years, using geometric properties of Markov processes developed in \cite{bakry2013analysis}.  Notably, paralleling the situation with convex optimisation \cite{BV04}, fast mixing of a Markov chain in high dimensions towards its invariant measure $\mu$ can be expected when \(H\) is convex \cite{dalalyan2017theoretical,DurmusMoulines2017,DurmusMoulines2019, rigetal}. It is also known that convexity -- or a related functional inequality satisfied by $\mu$, see e.g., \cite{CELSZ25} and references therein -- is essential for the polynomial mixing time of algorithms such as MALA in high dimensions, and that runtimes can scale exponentially in relevant parameters even for uni-modal target functions $H$ that are not globally convex, see \cite{BMNW23}.

\smallskip

In the context of statistical inference $-H$ is often a fit functional such as the (possibly penalised) log-likelihood function, and $\mu$ is the posterior measure arising from some prior on $\mathbb R^D$, see, e.g., \cite{Stuart2010, nickl2023bayesian}. The gradient $\nabla H$ then determines the Fisher information and thus the local average curvature near a `ground truth' parameter. In the recent contribution \cite{NicklWang2024_polynomialTimeLangevin} a related notion of \textit{gradient stability} was employed to establish high probability polynomial run-time  bounds for  Langevin type MCMC algorithms even when $H$ is not convex but $\mu$ is a posterior measure. Gradient stability has been verified in a variety of non-linear settings involving \textit{linear} PDEs \cite{NicklWang2024_polynomialTimeLangevin, nickl2023bayesian, BohrNickl2024, altmeyer2022polynomial}, but many complex models involve \textit{non-linear} PDEs such as the McKean--Vlasov or Navier-Stokes equations, and it is not generally clear how to check gradient stability there, or even how to derive expressions for the gradients themselves, as is relevant for implementation. 

The present article is a contribution to this circle of problems: we will describe a basic argument that allows to compute gradients of parameter-to-solution maps of non-linear PDEs based in essence on a Banach space version of the implicit function theorem. We illustrate how this methods works for two prototypical examples of non-linear PDEs introduced below, and use our findings to show that Bayesian inference algorithms for interacting particle models studied recently in \cite{mckv} are computable by warm start Langevin-type methods in high-dimensional settings. Along the way we  provide novel proof techniques to derive global injectivity estimates in non-linear inverse problems, in particular giving an alternative approach to establish some of the key analytical results for McKean--Vlasov systems in \cite{mckv}. Our ideas can further be related (e.g., via Theorem 15.3 in \cite{Duane1987}) to methods developed in \cite{DvKSZ23} that establish real analyticity and holomorphicity of  forward maps  appearing in inverse problems.

\subsection{General framework}\label{subsec:general-framework}

We focus our analysis on energies \(H\) of the form
\begin{equation}\label{eq:energy}
H(\theta) \coloneqq \frac{1}{2} \left( \sum_{i=1}^{N} | Y_i-\mathcal{G}(\theta)(X_i) |^{2} + \theta^{T} \Sigma^{-1} \theta \right)     
\end{equation} 
where \(\Sigma = \Sigma(N,D)\) is a $D \times D$ symmetric positive semi-definite matrix, \(\mathcal{G}\) some map from \(\mathbb{R}^{D}\) to the space $L^2(\mathcal X)$ of square-Lebesgue integrable functions defined on a bounded subset \(\mathcal X\) of the Euclidean space, and the \( (Y_i,X_i)_{i \leq N}\) are random variables valued in \mbox{\(\mathbb R \times \mathcal X\)}. We refer to \(\mathcal{G}\) as the `forward map', following the terminology of the inverse problems literature~\cite{Stuart2010}. Gibbs measures of this type arise naturally in Bayesian inference as posterior distributions with data coming from noisy point observations of \(\mathcal{G}(\theta)\) and a centred Gaussian prior with covariance matrix \(\Sigma\). More specifically, if we observe data \(Z_N=(Y_i,X_i)_{i \leq N}\) generated according to the model
\[
Y_i = \mathcal{G}(\theta)(X_i) + \varepsilon_i,, \quad \varepsilon_i \overset{\mathrm{iid}}{\sim} \mathcal{N}(0,1), ~X_i \overset{\mathrm{iid}}{\sim} \mathrm{Unif} (\mathcal X),
\]
and let \(\Pi\) be the distribution of the prior, we obtain the conditional `posterior' law
\begin{equation}\label{eq:posterior-intro}
\mathop{}\!\mathrm{d} \Pi (\theta | Z_N) \propto e^{\ell_N(\theta)} \mathop{}\!\mathrm{d} \Pi(\theta) = e^{-H(\theta)} \mathop{}\!\mathrm{d}\theta    
\end{equation}  
where \(\ell_N(\theta) = -\frac{1}{2} \sum_{i=1}^{N} | Y_i-\mathcal{G}(\theta)(X_i) |^{2} \) is the log-likelihood of the model. We denote by \(P_{\theta}^{N}\) the product measure describing the distribution of the data \mbox{\(Z_N\)} with parameter \(\theta\). We refer, e.g., to Chapter 1 in \cite{nickl2023bayesian} for a review of this standard framework. 

The notion of `gradient stability' discussed above involves identification of the gradient $\nabla \mathcal G_\theta$ of $\mathcal G$ at $\theta$ and the injectivity estimate 
\begin{equation}\label{gradstab}
\|\nabla \mathcal G_\theta(h)\|_{L^2(\mathcal X)} \gtrsim \|h\|_{\mathbb R^D}~~ \forall h \in \mathbb R^D,
\end{equation} see \cite[Sec.~3.1]{NicklWang2024_polynomialTimeLangevin} or \cite[Sec.~3.2] {nickl2023bayesian}. This requirement is closely related to the existence of the inverse Fisher information at $\theta$ in the underlying regression model. Gradients of forward maps $\mathcal G$ have been computed in a variety of examples involving non-linear parameter identification problems with \textit{linear} PDEs (such as Schr\"odinger equations \cite{N20, NicklWang2024_polynomialTimeLangevin},  Darcy's problem \cite{nickl2023bayesian, altmeyer2022polynomial}, non-Abelian $X$-ray transforms \cite{MNPCPAM, MonardNicklPaternain2021, BohrNickl2024}, diffusion problems \cite{GW25}) and also with initial value to solutions maps of non-linear PDEs such as reaction diffusion and Navier-Stokes equations (see \cite{bvm-time-evolution, KonenNickNSE}). These proofs however do not extend straightforwardly to the more difficult situation where one wishes to identify a general parameter of a \textit{non-linear} PDE, such as the interaction potential $W$ in a McKean--Vlasov equation (\ref{eq:mckean-vlasov-intro}) for example. The reason for the difficulty is that two non-linearities need to be dealt with simultaneously: one arising from the differential operator and one from the parameter to solution map. 

To explain our approach, we will carry out our analysis in general function spaces, viewing \(\theta\) as an element of a Banach space~\(\Theta\). The Euclidean space~\(\mathbb{R}^{D}\) can then be regarded as a linear subspace of~\(\Theta\) through an appropriate discretisation procedure (for instance, using a wavelet or Fourier basis/frame). Thus, given two Banach spaces \(X\) and \(Y\), we view \(\mathcal{G}\) as a map from \(\Theta\) to \(X\) and assume that there is an explicitly defined function $f$ of two variables, \(f:\Theta \times X \to Y\) such that $\mathcal G$ can be defined implicitly via
\begin{equation}\label{eq:fundamental-equation}
    f(\theta, \mathcal{G}(\theta)) = 0, \quad \theta \in \Theta.  
\end{equation} 
Such $f$ can accommodate a wide variety of problems --  we will illustrate the framework here with two non-linear parabolic PDE examples. We begin with a reaction-diffusion model in which the parameter \(\theta\) represents a Lipschitz reaction term \(R: \mathbb R \to \mathbb R\), and the nonlinear map \(\mathcal{G}(\theta)\eqqcolon u_R\) is defined as the unique solution $u_R=u$ to the nonlinear parabolic PDE
\begin{equation}\label{eq:reaction-diffusion-intro}
    \begin{cases}
     \partial_t u - \Delta u = R(u) &\text{on } (0,T) \times \mathbb{T}^{d} \\
     u(0) = \phi &\text{on } \mathbb{T}^{d}
\end{cases}
\end{equation} 
for some given initial condition \(\phi\) on the  torus \(\mathbb{T}^{d} \coloneqq \mathbb{R}^{d}/\mathbb{Z}^{d}\), and where $\Delta$ is the Laplacian. If we set $$f(R, u) = (\partial_t u - \Delta u - R(u), u(0)-\phi)$$ then the pairs $(R, u_R)$ are precisely the zeros of $f$ in (\ref{eq:fundamental-equation}) for appropriate choice of $\Theta, X,Y$. 

The second example is motivated by a surge of recent interest in statistical inference methods for interacting particle models, e.g., \cite{della2022nonparametric, amorino2024polynomial, pavliotis2025linearization, mckv}. Here one considers a non-linear PDE known as the McKean--Vlasov equation in which the parameter \(\theta\) represents an `interaction potential' denoted by \(W\), and \(\mathcal{G}(\theta)\eqqcolon\rho_W\) is defined as the unique solution $\rho_W=\rho$  to the non-linear and non-local Fokker-Planck equation (with $\nabla \cdot$ the divergence operator)
\begin{equation}\label{eq:mckean-vlasov-intro}
    \begin{cases}
\partial_t \rho = \Delta \rho + \nabla \cdot (\rho \nabla W \ast \rho) &\text{on \([0,T] \times \mathbb{T}^{d}\)}\\
     \rho(0)=\phi &\text{on \(\mathbb{T}^{d}\)}
\end{cases}
\end{equation} 
for some probability density \(\phi\) on \(\mathbb{T}^{d}\). Similar to above, we can take $$f(W, \rho) = (\partial_t \rho - \Delta \rho - \nabla \cdot (\rho \nabla W \ast \rho), \rho(0)-\phi)$$ in (\ref{eq:fundamental-equation}). The techniques we develop below extend well beyond the two examples discussed here and can, for instance, be applied to statistical inference problems occuring with Navier--Stokes equations \cite{NT24, KonenNickNSE}. 

A natural idea used commonly in PDE theory (e.g., Chapter III in \cite{AG07}) is to formally differentiate equation \eqref{eq:fundamental-equation} with respect to \(\theta\). This yields an identity
\begin{equation}\label{eq:derivative-identity}
D\mathcal{G}(\theta ) = -(D_{2}f(\theta, \mathcal{G}(\theta )))^{-1} \circ D_{1}f(\theta, \mathcal{G}(\theta ))   
\end{equation} 
that involves only partial derivatives of the map $f$ -- these can typically be computed without difficulty. Provided that \(f\) is continuously differentiable and that \(D_2f\) is a homeomorphism, the formula above can be made rigorous using the implicit function theorem in Banach spaces (e.g., \cite[Theorem 10.2.1]{dieudonne} and Theorem \ref{thm:abstract-derivative} below). Higher order differentiability of \(\mathcal{G}\) then directly follows from \eqref{eq:derivative-identity} and higher differentiability of \(f\). Let us summarise the main idea in the form of the following assertion. 
\begin{claim}
    Let \(f\) be such that \eqref{eq:fundamental-equation} holds. Assume moreoever that \(f\) is \(C^{k}\) in the Fréchet sense for some \(k \geq 1\), and that \(D_2f(\theta,\mathcal{G}(\theta)) : X \to Y\) is a homeomorphism for all \(\theta \in \Theta\). Finally, assume that for any \(\theta \in \Theta\), \(u=\mathcal{G}(\theta)\) is the unique solution in \(X\) to the equation \(f(\theta,u)=0\). Then, \(\mathcal{G}\) is \(C^{k}\) in the Fréchet sense and its derivative is given by \eqref{eq:derivative-identity}.    
\end{claim} 

We can heuristically conclude from this that verifying gradient stability (\ref{gradstab}) will entail an injectivity estimate for the map in (\ref{eq:derivative-identity}), which in turn reduces to injectivity of the partial derivative $D_1f$ of $f$. The injectivity of $D_2f$ on an appropriate domain is already an implicit requirement in the above formula for the derivative of $\mathcal G$. We also explain in Section \ref{stablin} how these ideas provide injectivity results and stability estimates for the non-linear map $\mathcal G$ via `pseudo-linearisation' identities commonly used in non-linear inverse problems \cite{MNPCPAM, PSU23}.

\subsection{Main results}

For the reaction diffusion model \eqref{eq:reaction-diffusion-intro}, we will show that the forward map \(R \mapsto u_R\) is continuously differentiable in the Fréchet sense, and that the linear map representing its derivative $Du_R$ at a point \(R\) can be computed by solving a linear parabolic PDE. This is summarised in the following result. A rigorous statement and its proof are provided in Theorem \ref{thm:linearisation-reaction-diffusion}.
\begin{claim}[Derivative of reaction-diffusion forward map] If \(d \leq 3\) and \(\phi \in H^{1}\), then for any \(R,H \in C^{2}_b(\mathbb R)\), the Fréchet derivative of \(u_R\) with respect to \(R\), \(Du_R[H] \coloneqq i_H =i\), is the unique solution to the linear parabolic PDE
         \[
    \begin{cases}
     \partial_t i - \Delta i - R'(u_R)i = H(u_R) &\text{on } (0,T) \times \mathbb{T}^{d} \\
     i(0) = 0 &\text{on } \mathbb{T}^{d}.
\end{cases}
         \]
\end{claim}

We now address the problem of computing the derivative of \(\mathcal{G}\) in the McKean--Vlasov model defined in~\eqref{eq:mckean-vlasov-intro}. Our general results will provide a  Fr\'echet derivative in function space (see Theorem \ref{thm:linearisation-mckean-vlasov}), but for the applications to sampling from high-dimensional Gibbs measures (\ref{eq:energy}) by gradient-based algorithms, let us first focus on the case where $W$ is restricted to lie in a $D$-dimensional approximation space of $\Theta$. Consider a trigonometric basis of \(L^2\) (see Section~\ref{subsec:approximation-spaces}), and for \(D \geq 2\) even, identify each vector in \(\mathbb{R}^{D}\) with the corresponding mean-zero function defined by its first $D$ Fourier coefficients. In this way, \(\rho_W\) is well-defined for \(W \in \mathbb{R}^{D}\). For any fixed \((t,x) \in [0,T] \times \mathbb{T}^{d}\), define the map \(\mathcal{G}^{t,x}\) by \(\mathcal{G}^{t,x}(W) \coloneqq \rho_W(t,x)\), \(W \in \mathbb{R}^{D}\). Then, the gradient of \(\mathcal{G}^{t,x}\) at a point \(W \in \mathbb{R}^{D}\) can be computed by solving a linear parabolic PDE, as stated in the following claim. Full details are provided in Theorem \ref{thm:linearisation-mckean-vlasov} and Lemma \ref{lem:discrete-derivatives-forward-map}.

\begin{claim}[Gradient of McKean--Vlasov forward map] Assume that the probability density \(\phi\) in \eqref{eq:mckean-vlasov-intro} is such that \(\phi \in H^{\beta}\) for some \(\beta \geq 3+d\). Then, for any \(H \in \mathbb{R}^{D}\) and \((t,x) \in [0,T] \times \mathbb{T}^{d}\), the gradient of \(\mathcal{G}^{t,x}\) with respect to \(W\) is given by \(H^{T}\nabla \mathcal{G}^{t,x} (W) = D \rho_W [H] (t,x)\), where \(D \rho_W [H] =u_H = u\) is the unique solution to the linear parabolic PDE
         \[
\begin{cases}
\partial_t u - \Delta u - \nabla \cdot (u \nabla W \ast \rho_W) - \nabla \cdot (\rho_W \nabla W \ast u) = \nabla \cdot (\rho_W \nabla H \ast \rho_W )&\text{on \([0,T] \times \mathbb{T}^{d}\)}\\
u(0)=0 &\text{on \(\mathbb{T}^{d}\)}.
\end{cases}
         \] 
\end{claim}

We now turn to the problem of generating random samples from the posterior distribution 
\[
\mathop{}\!\mathrm{d} \Pi ( W | Z_N) \propto e^{\ell_N(W)} \mathop{}\!\mathrm{d} \Pi(W)
\] defined earlier in \eqref{eq:posterior-intro}, for $\mathcal G(\theta)=\rho_W$ arising from the McKean--Vlasov model with $\mathcal X = [0,T] \times \mathbb T^d$. Using the preceding result and ideas from \cite{NicklWang2024_polynomialTimeLangevin}, we can verify gradient stability (\ref{gradstab}) and deduce that \(-\ell_N\) is locally strongly convex near the `ground truth parameter' $W_0$, on average under the discretisation scheme. This holds provided that the Fourier modes of \(\phi\) in \eqref{eq:mckean-vlasov-intro} do not decay too fast, which is a natural identifiability hypothesis in interacting particle models, as discussed in \cite{mckv}. More precisely, if \(Z_N\sim P_{W_0}^{N}\), then we have a lower bound on the smallest eigenvalue of the average Hessian
\[
\lambda_{\min} \left( \mathbb{E}_{W_0} [-\nabla^{2} \ell(W)] \right) \geq c_D > 0 
\] uniformly over potentials \(W\) in a ball around \(W_0\) whose size depends on the dimension $D$  and on the degree of identifiability of the model, governed by the decay of the Fourier modes of $\phi$. The constant \(c_D\) can here be taken to decay at worst polynomially in \(D\). This is the content of Theorem \ref{thm:local-average-curvature}. Then, following ideas from \cite{NicklWang2024_polynomialTimeLangevin,BohrNickl2024, nickl2023bayesian}, we introduce in Section \ref{thm:log-concave-approximation} a globally log-concave surrogate distribution \(\tilde{\Pi}_N\) approximating \(\Pi(\ \! \cdot \mid Z_N)\) in Wasserstein-2 distance. We sample from this surrogate using the Unadjusted Langevin Algorithm (ULA) defined by

\begin{equation}\label{eq:ULA-intro}
\vartheta_{k+1} =\vartheta_k+\gamma\nabla \log (\mathop{}\!\mathrm{d} \tilde{\Pi}_N(W))+\sqrt{2 \gamma} \xi_{k+1} 
\end{equation} 
where \(\xi_k \stackrel{\text { i.i.d. }}{\sim} \mathcal{N}\left(0, I_{D}\right)\) and $\gamma$ is a `stepsize' parameter. Each gradient step involves the (numerical) solution of the non-linear equation \eqref{eq:mckean-vlasov-intro} as well of the linearised parabolic PDE $D\rho_W[H]$, for which a variety of algorithms exist (see \cite{mckv} for references). We obtain a bound on the Wasserstein-2 mixing time of the chain $(\vartheta_k)$ that scales polynomially in \(N,D\) under `warm start'. For full details see Theorem \ref{thm:wasserstein-mixing-time}.

\begin{claim}[Polynomial mixing of ULA]
     Assume that the chain in \eqref{eq:ULA-intro} was started from \(\vartheta_0\) close enough to \(W_0 \in H^{\alpha}\) for some \(\alpha > 0\) large enough, and that the Fourier modes of \(\phi\) in \eqref{eq:mckean-vlasov-intro} do not decay too fast. Then, under further natural regularity assumptions, and for any desired accuracy \(\varepsilon >0\), there exists a step size \(\gamma=\gamma(\varepsilon,D,N)\) and an integer \(k_{\mathrm{mix}}\) growing at most polynomially in \(\varepsilon^{-1},D\) and \(N\), such that for all \(k \geq k_{\mathrm{mix}}\),
    \[
    \mathcal{W}_2^{2}(\mathcal{L}(\vartheta_k), \Pi(\ \! \cdot \mid Z_N)) \leq e^{-N^{d/(2(\alpha+1)+d)}} + \varepsilon 
    \] with \(P_{W_0}^{N}\)-probability tending to 1 exponentially fast.
\end{claim}
In the claim above, \(\mathcal{W}_2\) is the Wasserstein-2 distance defined by
\begin{equation}\label{eq:def-wasserstein}
\mathcal{W}_2^{2}(\mu,\nu)\coloneqq \inf_\pi \int_{\mathbb{R}^{D} \times \mathbb{R}^{D}} \left\| x - y \right\|^{2}_D \mathop{}\!\mathrm{d} \pi(x,y)    
\end{equation} 
where \(\left\| \cdot  \right\|_{D}\) is the Euclidean norm on \(\mathbb{R}^{D}\) and the infimum is over all couplings \(\pi\) of \(\mu\) and \(\nu\).

\subsection{Notation}

The set \(\mathbb{N}\) is defined as \(\mathbb{N} \coloneqq \left\{ 1,2,\dots  \right\} \) and $\mathbb Z$ denotes the set of all integers. Throughout, we take the \(d\)-dimensional flat torus \(\mathbb{T}^{d} \coloneqq \mathbb{R}^{d} / \mathbb{Z}^{d}, d \in \mathbb N\) as our spatial domain, equipped with the  induced Euclidean metric. We write \(A \lesssim B\) when there is a constant \(c>0\) such that \(A \leq cB\), where it is implicitly understood that the constant \(c\) only depends on fixed parameters. Then, \(A \gtrsim B\) means \(B \lesssim A\) and \(A \simeq B\) means both \(A \lesssim B\) and \(A \gtrsim B\). We will write \(A \lesssim_p B\) to indicate that the implicit constant \(c\) depends on \(p\). We adopt the standard convention in analysis that the symbol \(C\) denotes a positive constant whose value may vary from line to line.

For \(1 \leq p \leq +\infty\), let \(L^{p}(\mathbb T^{d})\), or \(L^{p}\) for short, be the usual periodic Lebesgue space of real valued functions. Let \(\mathcal{D}(\mathbb T^{d},\mathbb{C})\) be the space of complex-valued, \(C^{\infty}\) and \(1\)-periodic functions on \(\mathbb R^{d}\). Let then \(\mathcal{D}'(\mathbb{T}^{d},\mathbb{C})\) be the space of periodic distributions, i.e. the topological dual of the complex vector space \(\mathcal{D}(\mathbb T^{d},\mathbb{C})\). We define the Sobolev spaces \(H^{s}(\mathbb T^{d})\), or \(H^{s}\) for short, for any \(s \in \mathbb R\), as follows:
\[
H^{s} = \left\{ u \in \mathcal{D}'(\mathbb{T}^{d},\mathbb{C}) : \overline{\hat{u}_{k}} = \hat{u}_{-k} \text{ and } \sum_{k\in \mathbb Z^{d}} (1+\left| k \right|^{2})^s \left| \hat{u}_{k} \right|^{2} < +\infty  \right\} 
\]  
where \(\hat{u}_{k} \coloneqq \left\langle u, e^{2i\pi k \cdot x}\right\rangle \) is the \(k^{th}\) Fourier coefficient of the distribution \(u\) and the bracket denotes the usual duality pairing.
Equipped with the inner-product \(\left\langle u,v \right\rangle _{H^{s}} = \sum_{k\in \mathbb Z^{d}} (1 + |k|^{2})^{s} \hat{u}_{k} \overline{\hat{v}_{k}}\) , \(H^{s}\) is a real, separable Hilbert space. We further define the homogeneous Sobolev space \(\dot{H}^{s}\) as the closed subspace of \(H^{s}\) consisting of \(u \in H^{s}\) such that \(\hat{u}_0=0\). On \(\dot{H}^{s}\), we have the equivalent inner-product \(\left\langle u,v \right\rangle _{\dot{H}^{s}} = \sum_{k\in \mathbb Z^{d}} |k|^{2s} \hat{u}_{k} \overline{\hat{v}_{k}}\). More generally, for \(B\) a Banach space of functions defined on \(\mathbb{T}^{d}\), \(\dot{B}\) will denote the closed subspace of \(B\) consisting of mean-zero functions. In particular, we have \(\dot{L}^{2} = \dot{H}^{0}\), and we let \(\langle u,v\rangle \coloneqq \int u \bar{v}\) for \(u,v \in L^2\). For \(k \in \mathbb N\), the space \(W^{k,\infty}\) is the Sobolev space of functions in \(L^{2}\) with weak partial derivatives up to order \(k\) in \(L^{\infty}\).

For \(T>0\), let \(\mathcal X=[0,T]\times\mathbb T^d\) be the time–space cylinder equipped with the uniform probability measure \(\lambda=\lambda_T\) on \(\mathcal X\).  We denote by \(L^2(\mathcal X,\lambda)\) the space of  square–integrable functions on \(\mathcal X\). More generally, let \(B\) be a Banach space and \(1\le p\le\infty\).  We write \(L^p([0,T];B)\) for the Bochner space of strongly measurable maps \(t\mapsto H(t)\in B\) such that
\(t\mapsto\|H(t)\|_B\in L^p([0,T])\).  When \(B\) is a space of spatial functions on \(\mathbb T^d\) we often view such an \(H\) as a function \(H:\mathcal X\to\mathbb R\) with \(H(t,\cdot)\in B\) for almost every \(t\). For an integer \(m\ge 0\) the Sobolev space in time \(H^m([0,T];B)\) is defined by
\[
H^m([0,T];B)
=\Big\{H\in L^2([0,T];B)\;:\;\partial_t^j H\in L^2([0,T];B)\ \text{for }0\le j\le m\Big\},
\]
equipped with the norm \(\|H\|_{H^m([0,T];B)}^2 \;=\; \sum_{j=0}^m \big\|\partial_t^j H\big\|_{L^2([0,T];B)}^2\). We will often use the shorthand \(H^m_TB\) for \(H^m([0,T];B)\). For any two Banach spaces \(B_1, B_2\), we will endow the intersection \(B=B_1 \cap B_2\) with the sum norm \(\|\cdot \|_B = \|\cdot \|_{B_1} + \|\cdot \|_{B_2}\). Recall that \((B,\| \cdot \|_{B})\) then defines a Banach space, as long as \(B_1\) and \(B_2\) continuously embed into a common Hausdorff topological vector space (e.g. \(L^{2}([0,T],L^{2})\)). For \(X\) and \(Y\) two normed spaces, we write \(X \hookrightarrow Y\) when \(X\) is continuously embedded in \(Y\).

We will also repeatedly use the following observation from parabolic PDE theory: By choosing a suitable mollifying sequence, exactly as in the proof of \cite[Chapter~5.9, Theorem 4]{evans2010partial}, one shows that for any \(\gamma \geq 0\) and \(u \in  L^{2}([0,T];H^{\gamma+1}) \cap H^{1}([0,T];H^{\gamma-1})\), we have \(u \in C([0,T];H^{\gamma})\) and 
\begin{equation}\label{eq:trace-theorem-time}
    \sup_{0 \leq t \leq T} \left\| u(t) \right\|_{H^{\gamma}} \lesssim_{\gamma,d,T} \left\| u \right\|_{L^{2}([0,T];H^{\gamma+1})} + \left\| u \right\|_{H^{1}([0,T];H^{\gamma-1})}.
\end{equation} 

\section{Computing gradients in nonlinear PDE models}

In this section, we demonstrate how general techniques from non-linear functional analysis and PDE theory reviewed in the appendix apply to the two parabolic equations (\ref{eq:reaction-diffusion-intro}), (\ref{eq:mckean-vlasov-intro}) that we presented in the introduction. We show how to derive expressions for and prove stability of the gradient of $\mathcal G$ (as in (\ref{gradstab})) and we also explore in Subsection \ref{stablin} the related issue of proving global stability estimates for $\mathcal G$ itself by `pseudo-linearisation' ideas.

\subsection{Reaction-diffusion equation}

Let \(C^{k}_b(\mathbb R)\) be the Banach space of bounded \(k\)-times continuously differentiable functions on \(\mathbb R\) with bounded derivatives up to order \(k\), equipped with the norm \(\|f\|_{C^{k}} \coloneqq \sup_{x\in \mathbb{R}} \sum_{i=0}^{k} |f^{(i)}(x)| \). Fix \(\phi \in H^{1}\). For \(R \in C^{2}_b(\mathbb R)\), let \(u=u_R \in L^{2}([0,T];H^{2})\cap H^{1}([0,T];L^{2})\) denote the unique strong solution of
\begin{equation}\label{eq:reaction-diffusion}
    \begin{cases}
     \partial_t u - \Delta u = R(u) &\text{in } (0,T) \times \mathbb{T}^{d} \\
     u(0) = \phi &\text{in } \mathbb{T}^{d}
\end{cases}
\end{equation} 
Existence of the solution \(u_R\) and its uniqueness in \(C([0,T];L^2)\) follow from standard fixed point arguments, see e.g. \cite[p.536f]{evans2010partial}, letting the operator \(A\) in that reference assign \(u\) to \emph{strong} (instead of weak) solutions of the linear problem \(\partial_t v - \Delta v = R(u)\) with initial condition \(\phi\), see also \cite[Theorem 6]{bvm-time-evolution}. We omit the details. In applications, \(u_R\) describes the evolution of some chemical substance (with initial distribution \(\phi\)), where \(R(u)\) represents the amount of the substance being created or destroyed at each point in space and time. We assume $d \le 3$ to focus on the main ideas.

\begin{theorem}[Linearisation of the reaction-diffusion equation]\label{thm:linearisation-reaction-diffusion}
    Let \(d \leq 3\). The map \(\mathcal{G}: C^{2}_b(\mathbb R) \to L^{2}([0,T];H^{2})\cap H^{1}([0,T];L^{2})\) defined by \(\mathcal{G}(R) = u_R\) solving \eqref{eq:reaction-diffusion} is \(C^1\) in the Fréchet sense.    Moreover, for any \(R\in C^{2}_b(\mathbb R)\), its Fréchet derivative at $R$, $$D \mathcal{G}(R) :C^2_b(\mathbb R) \to L^{2}([0,T];H^{2})\cap H^{1}([0,T];L^{2}) $$ is given by the linear map \(D\mathcal{G}(R)[H] \eqqcolon i_H\) where  \(i=i_H\in L^{2}([0,T];H^{2})     \cap H^{1}([0,T];L^{2})\) is the unique strong solution of the following linear parabolic PDE
    \[
    \begin{cases}
     \partial_t i - \Delta i - R'(u_R)i = H(u_R) &\text{in } (0,T) \times \mathbb{T}^{d} \\
     i(0) = 0 &\text{in } \mathbb{T}^{d}.
\end{cases}
    \]
\end{theorem} 

\begin{proof}[Proof of Theorem \ref{thm:linearisation-reaction-diffusion}]
    The proof is a direct application of Theorem \ref{thm:abstract-derivative}; we first fit the notation to the setting of that theorem.
    Let
    \[
    E =C^{2}_b(\mathbb R),\quad  F = L^{2}([0,T];H^{2})\cap H^{1}([0,T];L^{2}),\quad G = L^{2}([0,T];L^{2})\times H^{1}
    \]
    and define the map
\begin{equation}\label{eq:f-reaction-diffusion}
\begin{array}{rcl}
f: & E \times F &\longrightarrow \hspace{56pt} G\\[6pt]
   & (R,u) &\longmapsto \left( \partial_t u - \Delta u - R(u), u(0)-\phi \right).
\end{array}
\end{equation}
The function \(f\) is well-defined  by \eqref{eq:trace-theorem-time} and since \(\phi \in H^{1}\) and \(R \in C^{2}_b(\mathbb{R})\). Let us for the sake of illustration show carefully that \(f\) is \(C^{1}\) on \(E \times F\). Consider auxiliary maps \(g:F \to G\) and \(h: E \times F \to L^{2}([0,T];L^{2})\) defined by 
\[
g(u) = (\partial_t u - \Delta u, u(0)), \quad h(R,u) = R(u).
\] 
Then, \(f(R,u)=g(u)+ (-h(R,u),0) + (0,-\phi)\). The map \(g\) is clearly linear, and it is moreover continuous, using again \eqref{eq:trace-theorem-time}. Hence, \(g\) is smooth in the Fréchet sense by Lemma \ref{lem:mult-is-smooth}. Since \((0,-\phi)\) is just a constant, it suffices to prove that \(h\) is \(C^{1}\). For any \(R \in E\), we note 
\[
\sup_{u \in F} \| h(R,u) \|_{L^{2}_TL^{2}}  \leq \sqrt{T} \| R \|_{C^{0}}.
\]
Since \(h\) is linear in \(R\), we conclude that \(h\) is continuous in the first component, uniformly over \(u \in F\). From this we also conclude that \(h\) is \(C^{\infty}\) in the first component, and \(D_1h(R,u)[H] = H(u)\). Next, for any \(u \in F\) and \(u_n \to u\) in \(F\), we obtain by the fundamental theorem of calculus
\[
\| R(u_n) - R(u) \|_{L^{2}_{T}L^{2}} \leq \| R \|_{C^{1}} \| u_n - u \|_{L^{2}_{T}L^{2}} \to 0.
\]    
Therefore, \(h\) is jointly continuous on \(E \times F\). Let us now compute the derivative of \(h\) with respect to \(u\). Using the fundamental theorem of calculus twice, we obtain for any \(s \in F\),
\[
\left| R(u+s) - R(u) - R'(u)s \right| \leq \frac{1}{2} \| R \|_{C^{2}} |s|^{2} \quad \text{ a.e. in } [0,T] \times \mathbb{T}^{d}.
\] 
Integrating over time and space, we obtain
\[
\| R(u+s) - R(u) - R'(u)s \|_{L^{2}_TL^{2}} \leq \frac{1}{2} \| R \|_{C^{2}} \| s^{2} \|_{L^{2}_TL^{2}} \leq  \frac{1}{2} \| R \|_{C^{2}} \| s \|_{L^{\infty}_TL^{2}} \| s \|_{L^{2}_TL^{\infty}}.
\] 
Furthermore, we obtain \(\| s \|_{L^{\infty}_TL^{2}} \lesssim \| s \|_{F}\) by \eqref{eq:trace-theorem-time}, and \(\| s \|_{L^{2}_TL^{\infty}} \lesssim \| s \|_{F}\) by Sobolev embeddings (using that \(d \leq 3\)). Since \(s \mapsto R'(u) s\) is bounded linear, we conclude that \(h\) is differentiable in the second component with derivative \(D_2h(R,u)[s] = R'(u) s\). To conclude that \(h\) is \(C^{1}\), it remains to show that \(D_2h\) is jointly continuous in \(E \times F\). Let \((R,u) \in E \times F\) with \(R_n \to R\) in \(E\) and \(u_n \to u\) in \(F\). Then, for any \(s \in F\) with \(\| s \|_{F} \leq 1\), we obtain
\begin{align*}
    \| R_n'(u_n)s - R'(u)s \|_{L^{2}_TL^{2}} &\leq \| s \|_{L^{\infty}_T L^{2}} \| R_n'(u_n) - R'(u) \|_{L^{2}_TL^{\infty}}  \\
    &\lesssim \| s \|_{F} \left( \| R_n'(u_n) - R'(u_n) \|_{L^{2}_TL^{\infty}} + \| R'(u_n) - R'(u) \|_{L^{2}_TL^{\infty}} \right)\\
    & \leq \sqrt{T} \| R_n - R \|_{E} + \| R \|_{C^{2}} \| u_n - u \|_{F} \xrightarrow[n \to \infty]{} 0
\end{align*}  
using again the fundamental theorem of calculus to obtain that \(R'\) is Lipschitz, and the Sobolev embedding \(H^{2} \hookrightarrow L^{\infty}\) when \(d \leq 3\). This concludes the proof that \(h\) is \(C^{1}\), and hence that \(f\) is \(C^{1}\) as well. Let us note right away that, by the rules of differentiation for multilinear maps, we have for any \(R,H \in E\) and \(u,s\in F\),   
\[
D_1f(R,u)[H] = -(H(u),0), \quad D_2f(R,u)[s] = (\partial_t s - \Delta s - R'(u)s, s(0)).
\] 

We now check the final assumptions of Theorem \ref{thm:abstract-derivative}. For any \(R \in E\), \(u_R \in F\) (\(R \in C^2_b(\mathbb R)\) is globally Lipschitz), \(f(R,u_R)=0\) by definition of \(f\), and \(u_R\) is unique in \(C([0,T];L^{2})\) and hence in \(F \hookrightarrow C([0,T];L^{2})\) by \eqref{eq:trace-theorem-time}. Finally, given the form of \(D_2f\) in the display above, \(D_2f(R,u_R)\) defines a standard linear parabolic initial value problem whose theory is well known. In particular, \(D_2f(R,u_R) : F \to G\) is a bijection with bounded inverse by the standard existence, uniqueness and regularity theory for linear parabolic equations, see e.g. \cite[Chapter 7]{evans2010partial}. It is linear by definition, and hence continuous by the open mapping theorem. We can therefore apply Theorem \ref{thm:abstract-derivative} to conclude that \(\mathcal{G}\) is \(C^{1}\) in the Fréchet sense, and that its derivative is given by the formula in the statement of the theorem. 
\end{proof}

\begin{remark}[Stability of the gradient] \normalfont
    As already noted in the introduction, a key property of the gradient for statistical identifiability of the model is `gradient stability' (\ref{gradstab}). A complete analysis will be carried out in the McKean--Vlasov model below. For the reaction-diffusion model, we restrict ourselves to the following sketch of some ideas, to be investigated in detail in future work. The derivative in Theorem \ref{thm:linearisation-reaction-diffusion} can be written as
    \[
    D \mathcal{G} (R) [H] = (D_2 f(R,u_R))^{-1} [ H(u_R) ].
    \] 
Since \(D_2 f(R, u_R)\) is a linear homeomorphism, its inverse can be lower bounded in suitable norms, using interpolation estimates. The main injectivity challenge therefore stems from the term \(H(u_R)\). To show that \(H\) must vanish when \(\| H(u_R) \|^2_{L^{2}_{T}L^{2}} = \int_0^T\int_{\mathbb T^d} |H(u_R(t,x))|^2 \mathop{}\!\mathrm{d}x \mathop{}\!\mathrm{d}t\) does, one needs to control the range \(\{\, u_R(t,x) : (t,x) \in \mathcal{X} \,\}\). This can, for instance, be achieved via a small-time argument: since \(t \mapsto u_R(t)\) is continuous, for sufficiently small times the range of \(u_R\) remains close to that of the initial condition \(\phi\). By prescribing the range of \(\phi\), one may then hope to recover \(H\) on this set.
\end{remark}

\subsection{McKean--Vlasov equation}\label{sec:application-mckv}

Throughout this section, for $W \in \dot{W}^{2,\infty}$ a \textit{mean zero} interaction potential, \(\rho_{W}\) will denote the unique strong solution of a McKean--Vlasov non-linear and non-local parabolic PDE defined by
\begin{align}\label{eq:mckv}
  \partial_t \rho  &= \Delta \rho + \nabla \cdot (\rho \nabla W \ast \rho) \text{ on } [0,T]\times \mathbb{T}^{d}\\
  \rho(0)&= \phi \notag,
\end{align} where \(\phi\) is some probability density on \(\mathbb{T}^{d}\) modelling the initial state of the system. For background on how this equation arises as the mean-field limit of diffusing interacting particle systems, we refer the reader to the introduction of \cite{mckv} and references therein. To ensure sufficient regularity of $\rho_W$, we will employ the following assumption on $\phi$.
 
\begin{assumption}\label{assumption:regulatity-mckv}
    The initial condition \(\phi\) lies in \(H^{\beta }\) for some even integer \(\beta \geq 3 + d\).
\end{assumption}

Existence, uniqueness and regularity of \(\rho_W\) is guaranteed for $W \in \dot{W}^{2,\infty}$, under Assumption \ref{assumption:regulatity-mckv}:  in particular, the unique solution \(\rho_W\) to \eqref{eq:mckv} then lies in \(L^{2}([0,T];H^{\beta+1})\cap H^{1}([0,T];H^{\beta-1})\). We refer to \cite{chazelle_noisy_consensus}, \cite{armaGP}, \cite[Theorem 3.4]{Gvalani_thesis}, and \cite[Theorem 5]{mckv} for proofs supporting these claims. That $\beta$ is an \textit{even} integer is required only for convenience, as it streamlines some of the parabolic regularity proofs.

\subsubsection{Properties of the map \texorpdfstring{\(f\)}{f}}

Let us first fix some notation. We consider spaces 
\begin{equation}\label{keydefinition}
E :={\dot W}^{2,\infty}, ~F = F_\beta := L^{2}([0,T];H^{\beta+1}) \cap H^{1}([0,T];H^{\beta-1}), ~ G =G_\beta:= L^{2}([0,T];H^{\beta-1}) \times H^{\beta}.
\end{equation}
where we recall that \(F\) is endowed with the sum norm \(\| \cdot  \|_{L^{2}_{T}H^{\beta+1}} + \| \cdot \|_{H^{1}_{T}H^{\beta-1}}\). Then, define \(f : E \times F \to G\) and its components \(f_1, f_2\) by
\begin{equation}\label{eq:f-mckean-vlasov}
    f(W,\rho ) \coloneqq \left( \partial _{t}\rho - \Delta \rho - \nabla \cdot \left( \rho \nabla W \ast \rho \right), \rho(0)-\phi \right) \eqqcolon (f_1(W,\rho), f_2(W,\rho)).
\end{equation} 
Let us argue that \(f\) is well-defined on its domain. Owing to \eqref{eq:trace-theorem-time} and since \(\rho \in F\), \(\rho(0)\) is well-defined and is in \(H^{\beta}\). Then, let \(\mathcal{T}\) be the trilinear operator defined by
\begin{equation}\label{eq:def-Tcal}
\mathcal{T}(r,V,s) \coloneqq \nabla \cdot \left( r \nabla V \ast s \right). 
\end{equation} 
The following lemma proves that \(\mathcal{T}(r,V,s)\) is in \(L^{2}([0,T];H^{\beta-1})\) when \(V \in E\) and \(r,s \in F\).
\begin{lemma}\label{lem:trilin-bounded}
    For \(V \in E\), \(r,s \in F\) and integer \(\beta\geq 0\), 
    \[
    \left\| \mathcal{T}(r,V,s) \right\|_{L^{2}([0,T];H^{\beta-1})} \leq C \left\| r \right\|_{F} \left\| V \right\|_{E} \left\| s \right\|_{F}
    \] where \(C = C(\beta,d) > 0\).
\end{lemma}

\begin{proof}
     Lemma \ref{lem:trilin-op-T-bound} equation \eqref{eq:trilin-leibniz} with \(k=\beta\) and \(l=0\) gives     
\[
\left\| \mathcal{T}(r,V,s) \right\|_{L^{2}([0,T];H^{\beta-1})} \lesssim_{\beta ,d} \sup_{0 \leq t \leq T} \left\| r \right\|_{H^{\beta}} \left\| s \right\|_{L^{2}([0,T];H^{\beta})} \left\| \nabla V \right\|_{L^{\infty}}.
\]
By definition of the norms on \(E\) and \(F\), \(\left\| \nabla V \right\|_{L^{\infty}} \leq \left\| V \right\|_{E}\) and \(\left\| s \right\|_{L^{2}([0,T];H^{\beta})} \leq \left\| s \right\|_{F}\). Finally, using \eqref{eq:trace-theorem-time} we obtain \(\sup_{0 \leq t \leq T} \left\| r \right\|_{H^{\beta}} \lesssim \left\| r \right\|_{F}\), which concludes the proof. 

\end{proof}

From the lemma above we deduce that \(T(\rho,W,\rho)\) is in \(L^{2}([0,T];H^{\beta-1})\) and therefore \(f(W,\rho )\) is well-defined as an element of \(G\). We now turn to proving that \(f\) has the properties required in the framework of Section \ref{subsec:general-framework}. 

\begin{proposition}[Derivatives of \(f\)]\label{prop:derivatives-f}
    For any integer \(\beta \geq 0\), the map \(f:E \times F \to G\) defined in \eqref{eq:f-mckean-vlasov} is \(C^{\infty}\) in the Fréchet sense. The first derivative of \(f\) is given by, for any \(W,H \in E\) and \(\rho, s \in F\):
\begin{align*}
    &D_1 f(W, \rho) [H] = \left( - \nabla \cdot \left( \rho \nabla H \ast \rho \right), 0 \right)  \\
    &D_2 f(W, \rho) [s] \ = \left( \partial_t s - \Delta s - \nabla \cdot \left( s \nabla W \ast \rho  \right) - \nabla \cdot \left( \rho \nabla W \ast s \right), s(0) \right).
\end{align*} 
The second derivative of \(f\) is given by, for any \(W,H_1,H_2 \in E\) and \(\rho,s_1,s_2 \in F\): 
\begin{align*}
    &D_{1,1}^{2}f(W,\rho) [H_1,H_2] = (0,0) \\
    &D_{1,2}^{2}f(W,\rho) [H_1,s_2] = \left( -\nabla \cdot \left( s_2 \nabla H_1 \ast \rho \right) -\nabla \cdot \left( \rho \nabla H_1 \ast s_2 \right) , 0\right)   \\
    &D_{2,2}^{2}f(W,\rho) [s_1,s_2] = \left( -\nabla \cdot \left( s_1 \nabla W \ast s_2 \right) -\nabla \cdot \left( s_2 \nabla W \ast s_1 \right),0\right),
\end{align*} and \(D_{2,1}^{2}f=D_{1,2}^{2}f\).
\end{proposition}

\begin{proof}
    See appendix \ref{appendix:further-technical-results}.
\end{proof}

The general form of the derivative of $\mathcal G$ in \eqref{eq:derivative-identity} requires as an essential ingredient the existence of an inverse of $D_2f$ between appropriate spaces. Inspection of the previous proposition suggests that for non-linear parabolic equations such as those considered here, $(D_2f)^{-1}$ is  essentially the solution operator of a \textit{linear} parabolic PDE, so that standard techniques apply. It is worth noting that existence of an appropriate inverse is less immediate in the corresponding \textit{steady-state} equation because of the presence of a zeroth-order coefficient in the resulting (now elliptic) differential operator.

\begin{proposition}\label{prop:D2f-homeomorphism}
    Let \(\phi\) satisfy Assumption \ref{assumption:regulatity-mckv} for some \(\beta \geq 3 + d\). For any \(W \in \dot{W}^{2,\infty}\), recall that \(\rho_W\) denotes the unique solution to \eqref{eq:mckv}. Then, the map
    \[
    D_2f(W,\rho _W) : L^{2}([0,T];H^{\beta+1}) \cap H^{1}([0,T];H^{\beta-1}) \to L^{2}([0,T];H^{\beta-1}) \times H^{\beta}
    \] from Proposition \ref{prop:derivatives-f} is a linear homeomorphism.
\end{proposition}

\begin{proof}[Proof of Proposition \ref{prop:D2f-homeomorphism}]
    Fix \(W \in E\), let \(B \coloneqq D_2f(W,\rho _W)\) and recall, with the notations of Appendix \ref{appendix:parabolic-regularity}, that 
    \[
    B u = \left( \partial_t u - \Delta u - \nabla \cdot \left( u \nabla W \ast \rho_W  \right) - \nabla \cdot \left( \rho_W \nabla W \ast u \right), u(0) \right) = \left( \left( \partial_t - \mathcal{L}_W \right)u, u(0)  \right).
    \]
    Theorem \ref{thm:existence-linPDE} combined with Theorem \ref{thm:linPDE-higher-regularity} then gives that for any \(g \in L^{2}([0,T];H^{\beta-1})\) there exists a unique \(u \in L^{2}([0,T];H^{\beta+1}) \cap H^{1}([0,T];H^{\beta-1})\) such that \(Bu = (g,0)\). To construct the full inverse, let \((g,h) \in L^{2}([0,T];H^{\beta-1}) \times H^{\beta}\), and recall that by the trace theorem, \cite[Theorem I.3.2]{lions}, there exists  \(v \in L^{2}([0,T];H^{\beta+1}) \cap H^{1}([0,T];H^{\beta-1})\) such that \(v(0)=h\) (e.g., \(v\) solves the heat equation with initial condition \(h\)). Choose such \(v\) and let \(w_v\) be the unique solution to 
    \[
    Bw_v = (g - \left( \partial_t - \mathcal{L}_W \right)v,0)
    \] which is well-defined by Theorem \ref{thm:existence-linPDE}. Then, let \(u \coloneqq v + w_v\) and notice that 
    \[
    Bu = Bv + Bw_v = (g,h).
    \] Note that \(u\) is unique because if \(u_1\) and \(u_2\) are such that \(Bu_1 = Bu_2 = (g,h)\) then \(B(u_1 - u_2) = (0,0)\) which then implies \(u_1=u_2\) using Theorem \ref{thm:existence-linPDE}. We therefore have constructed an inverse \(B^{-1}(g,h) \coloneqq u\), and we now show that the inverse is bounded. Using again the representation \(u=v+w_v\), where we fix \(v\) to be the heat flow with initial condition \(h\), the regularity estimate in Theorem \ref{thm:linPDE-higher-regularity}, standard regularity estimates for the heat equation, Lemma \ref{lem:trilin-bounded}, and \cite[Theorem 5]{mckv} yield
    \begin{align*}
        \left\| B^{-1}(g,h) \right\|_{F}=\left\| u \right\|_{F} &\leq \left\| v \right\|_{F} + \left\| w_v \right\|_{F}\\
        &\lesssim \left\| h \right\|_{H^{\beta}} + \left\| g - \left( \partial_t - \mathcal{L}_W \right)v \right\|_{L^{2}([0,T];H^{\beta-1})}\\
        &\lesssim \left\| h \right\|_{H^{\beta}}+\left\| g \right\|_{L^{2}([0,T];H^{\beta-1})} + \left\| v \right\|_{F} \lesssim \left\| (g,h) \right\|_{G} + \left\| h \right\|_{H^{\beta}} \leq C \left\| (g,h) \right\|_{G}
    \end{align*} with \(C=C(d,T,\left\| W \right\|_{E},\left\| \phi \right\|_{H^{\beta}})\). This proves that \(B^{-1}\) is bounded and therefore, by the open mapping theorem, the proposition is proved. 
\end{proof}

\subsubsection{Global stability estimates and pseudo-linearisation identity}\label{stablin}

Before turning to the computation of the linearisation, we briefly examine how equation \eqref{eq:fundamental-equation} can be useful to derive global stability estimates for the forward map \(\mathcal{G}: \Theta \to X\), where $\Theta, X$ are Banach spaces -- in fact, the arguments to follow work as long as $\Theta$ is an open convex subset of a Banach space. In the notation of Section \ref{subsec:general-framework} we wish to obtain inequalities of the form
\begin{equation}\label{eq:general-stability-estimate}
    \| \theta_2 - \theta_1 \|_1 \lesssim \| \mathcal{G}(\theta_2) - \mathcal{G}(\theta_1)\|_2
\end{equation} for an appropriate choice of norms \(\| \cdot  \|_1,\| \cdot  \|_2\). The implicit multiplicative constant is required to be uniform over all \(\theta_1, \theta_2 \in \Theta\) such that \(\| \theta_i \|_{\mathcal{R}} \leq M\), \(i\in\{1,2\}\), where \(\| \cdot  \|_{\mathcal{R}}\) is an appropriate regularisation norm, e.g. \(\mathcal{R} = H^{a}\), \(a\geq0\). Let us assume that \(f\) is \(C^{1}\) on \(\Theta \times X\) so that the fundamental theorem of calculus in Banach spaces, e.g. \cite[Proposition A.2.3]{LiuRockner_introSPDEs}, applied to \(u \mapsto f(\theta_1+(\theta_2-\theta_1)u, \mathcal{G}(\theta_1))\) combined with the chain rule, e.g. \cite[Theorem 8.2.1]{dieudonne}, yields
\[
f(\theta_2,\mathcal{G}(\theta_1)) - f(\theta_1,\mathcal{G}(\theta_1)) = \int_{0}^{1} D_1f(\theta_u, \mathcal{G}(\theta_1)) [\theta_2 - \theta_1] \mathop{}\!\mathrm{d}u
\] 
where \(\theta_u=\theta_1 + (\theta_2-\theta_1)u\) lies on the line segment connecting $\theta_1$ to $\theta_2$. Similarly, we obtain
\[
f(\theta_2,\mathcal{G}(\theta_2)) - f(\theta_2,\mathcal{G}(\theta_1)) = \int_{0}^{1} D_2f(\theta_2, \mathcal{G}_u) [\mathcal{G}(\theta_2) - \mathcal{G}(\theta_1)] \mathop{}\!\mathrm{d}u
\] 
where \(\mathcal{G}_u=\mathcal{G}(\theta_1)+(\mathcal{G}(\theta_2)-\mathcal{G}(\theta_1))u\). Summing up these two identities and using that \(f(\mathcal{G}(\theta),\theta)=0\) for any \(\theta \in \Theta\) (from \eqref{eq:fundamental-equation}), we deduce
\begin{equation}\label{eq:ftc-identity}
    \int_{0}^{1} D_2f(\theta_2, \mathcal{G}_u) [\mathcal{G}(\theta_2) - \mathcal{G}(\theta_1)] \mathop{}\!\mathrm{d}u = - \int_{0}^{1} D_1f(\theta_u, \mathcal{G}(\theta_1)) [\theta_2 - \theta_1] \mathop{}\!\mathrm{d}u.
\end{equation} 
Using that the Bochner integral commutes with bounded linear maps (see e.g. \cite[Proposition A.2.2]{LiuRockner_introSPDEs}), we can rewrite the previous identity as 
\begin{equation}\label{eq:pseudo-lin}
\mathcal{A} (\mathcal{G}(\theta_2) - \mathcal{G}(\theta_1)) = \mathcal{B} (\theta_2 - \theta_1)    
\end{equation} 
where we defined the averaged linear operators \(\mathcal{A}: X \to Y \) and \(\mathcal{B}:\Theta \to Y\) as 
\[
\mathcal{A} = \int_{0}^{1} D_2f(\theta_2, \mathcal{G}_u) \mathop{}\!\mathrm{d}u, \quad \mathcal{B} = - \int_{0}^{1} D_1f(\theta_u, \mathcal{G}(\theta_1)) \mathop{}\!\mathrm{d}u.
\]  
In the context of PDE inverse problems, this relation implies that a change in the parameter, \(\theta_2 - \theta_1\), is mapped to a change in the observation, \(\mathcal{G}(\theta_2)-\mathcal{G}(\theta_1)\), via the \emph{linear} PDE \eqref{eq:pseudo-lin}. Following the terminology of \cite[Chapter 13]{PSU23}, one may refer to \eqref{eq:pseudo-lin} as a \emph{pseudo-linearisation identity} -- see Remark \ref{remark:stab-estimate-linear-PDE} for how this applies to the McKean--Vlasov equation. In particular \eqref{eq:pseudo-lin} suggests that to obtain the stability estimate \eqref{eq:general-stability-estimate}, it suffices to find a third norm \(\left\| \cdot \right\|_{3}\) such that \(\| \mathcal{A}u \|_{3} \lesssim \| u \|_{2}\) and \( \| \mathcal{B}v \|_{3} \gtrsim \| v \|_{1} \). The upper bound for \(\mathcal{A}\) is generally straightforward since \(D_2f(\theta_2,\mathcal{G}_s)\) is itself already a bounded operator from \(X\) to \(Y\). The lower bound for \(\mathcal{B}\) is more delicate and is essential for the stability of the map $\theta \mapsto \mathcal G(\theta)$. 

\begin{remark}\label{rem:choice-of-path} \normalfont
    In principle, one could apply the fundamental theorem of calculus along \textit{any} \(C^1\) path from \((\theta_1,\mathcal{G}(\theta_1))\) to \((\theta_2,\mathcal{G}(\theta_2))\) to obtain an identity similar to \eqref{eq:ftc-identity}, as long as \(f\) remains \(C^1\) on the image of the path. Note that the piecewise \(C^{1}\) path chosen above works particularly well for problems where the parameter \(\theta\) appears linearly in the map \(f\). Indeed, \(D_1f\) is then constant in \(\theta\) so that \(\mathcal{B}\) simplifies to \(\mathcal{B} = - D_1f(\mathcal{G}(\theta_1))\). However, other choices of paths could be more adapted to specific problems. 
\end{remark}

As an illustration of these general ideas, we obtain a new proof of the stability estimate \cite[Theorem 2]{mckv} avoiding the fixed point iteration used there. To deal with identifiability issues, we follow \cite{mckv} and introduce the following assumption on the initial condition \(\phi\). 
\begin{assumption}\label{assumption:phi}
    \begin{enumerate}
    \item[]
    \item \(\phi\) is strictly positive on \(\mathbb T^{d}\) with \({\inf_{x \in \mathbb T^d}\phi(x)} \geq \phi_{\mathrm{min}}>0.\) 
    \item There exists \(c_{\ast}>0\) such that for all \(k\in \mathbb{Z}^{d}\) and some \(\zeta > \beta + d/2\)
    \begin{equation}\label{eq:fourier-decay-phi}
        |\widehat{\phi}_k| \geq c_{\ast} |k|^{-\zeta}.
    \end{equation}      
\end{enumerate}
\end{assumption}
Densities satisfying Assumptions \ref{assumption:phi} and \ref{assumption:regulatity-mckv} for any choice of integer $\beta>0$ and \(\zeta > \beta + d/2\) do exist: this is discussed in \cite[Example 3]{mckv} for $\zeta=2m, m \in \mathbb N,$ and extends to arbitrary $\zeta>\beta+d/2$ after replacing the Laplace-type densities in \cite{mckv} by periodised probability kernels of appropriate Bessel-potentials as given on p.132 in \cite{S70}. An interpretation of this hypothesis is that the initial state of the interacting particle system is sufficiently different from the uniform equilibrium state (see \cite{armaGP}) so that the observed dynamics are informative. If the dynamics are started with $\phi=1$ then all its Fourier modes except the first vanish and (\ref{eq:fourier-decay-phi}) fails. But such $\phi$ is a steady state of the dynamics \eqref{eq:mckv} for all $W$, which implies $\rho_W(t)=1$ for all $t$ and \(W\) and therefore forbids a stability estimate -- see \cite{mckv} for more discussion.

\begin{theorem}[Stability estimate]\label{thm:mckv-stability-estimate}
    Let \(\phi\) satisfy Assumption \ref{assumption:regulatity-mckv} for some \(\beta \geq 3 + d\) as well as Assumption \ref{assumption:phi} for some \(\zeta > \beta + d/2\). For any \(W \in \dot{W}^{2,\infty}\), let \(\rho_W\) denote the unique solution to the PDE \eqref{eq:mckv} with initial condition $\phi$. Consider the space \(\dot{E}_K\) of mean-zero trigonometric polynomials of degree at most \(K\) from \eqref{eq:def-E_K}. Fix \(M>0,K\geq 1\) and let \(W_1 \in \dot{W}^{2,\infty}\), \(W_2 \in \dot{E}_K\) such that \(\left\| W_1 \right\|_{W^{2,\infty}}, \left\| W_2 \right\|_{W^{2,\infty}} \leq M\). Then, there exists a constant \(S\) depending on \(d,T,M,c_{\ast},\phi_{\mathrm{min}},\| \phi \|_{H^{d}}\) (but not on \(K\)) such that 
    \[
    \left\| W_2-W_{1} \right\|_{L^{2}} \leq \left\| W_1 - W_{1,K} \right\|_{L^{2}} + S K^{3\zeta/2} \left(\left\| \rho_{W_2} - \rho_{W_1} \right\|_{L^{2}([0,T];L^{2})} + \left\| \rho_{W_2} - \rho_{W_1} \right\|_{H^{1}([0,T];H^{-2})} \right)
    \] where \(W_{1,K}\) is the projection of \(W_1\) onto the space \(\dot{E}_K\).
\end{theorem} 

\begin{proof}[Proof of Theorem \ref{thm:mckv-stability-estimate}] 
    Define two \(C^{1}\) paths \(W(u)\) and \(\rho(u)\) by 
\[
W(u) \coloneqq W_1 + (W_2-W_1)\,u,
\quad
\rho(u) =\rho_{W_1} + (\rho_{W_2}-\rho_{W_1})\,u
\] for all \(0 \leq u \leq 1\). 
By the fundamental theorem of calculus in Banach spaces (e.g. \cite[Proposition A.2.3]{LiuRockner_introSPDEs} with Banach space \(X\) in that reference equal to \(L^{2}([0,T];H^{\beta-1}) \times H^\beta\))), we obtain, as in \eqref{eq:ftc-identity},
\begin{align}
    \int_{0}^{1} D_2 f(W_2,\rho(u)) [\rho_{W_2}-\rho_{W_1}] \mathop{}\!\mathrm{d} u &= -\int_{0}^{1} D_1 f(W(u),\rho_{W_1}) [W_2-W_1] \mathop{}\!\mathrm{d} u \label{eq:pre-stab-estimate}\\
    &= (\nabla \cdot (\rho_{W_1} \nabla (W_2 - W_1) \ast \rho_{W_1}), 0) \nonumber
\end{align} 
where \(D_i f, i=1,2\) are given in Proposition \ref{prop:derivatives-f}. From this we deduce
\begin{equation}\label{eq:pre-stab-estimate-projected}
    \nabla \cdot (\rho_{W_1} \nabla (W_2 - W_1) \ast \rho_{W_1}) = A (\rho_{W_2} - \rho_{W_1}) \text{ in } L^{2}([0,T];H^{\beta-1})
\end{equation} 
with \(A (\rho_{W_2} - \rho_{W_1}) \coloneqq \int_{0}^{1} D_2 f_1(W_2,\rho(u)) [\rho_{W_2}-\rho_{W_1}]  \mathop{}\!\mathrm{d} u\).
We first find an upper bound for the right-hand side and then prove a lower bound for the left-hand side using a deconvolution argument from \cite{mckv}.

The Bochner integral in the definition of \(A\) converges in \(L^{2}([0,T];H^{\beta-1})\) which continuously embeds into \(L^{2}([0,T];H^{-2})\), so that the integral also converges in \(L^{2}([0,T];H^{-2})\). Therefore, we use the triangle inequality, Lemma \ref{lem:trilin-op-T-bound} equation \eqref{eq:trilin-leibniz} with \(k=l=0\) and \(\nu=1\), the Sobolev embedding, and \cite[Theorem 5]{mckv} to obtain
\begin{align*}
    \| A(\rho_{W_2}-\rho_{W_1}) \|_{L^{2}([0,T];H^{-2})} &\leq \int_{0}^{1} \| D_2 f_1(W_2,\rho(u))[\rho_{W_2}-\rho_{W_1}]  \|_{L^{2}([0,T];H^{-2})} \mathop{}\!\mathrm{d}u\\
    &\lesssim_{d,T} \int_{0}^{1} \| \rho(u) \|_{L^{\infty}_TL^{\infty}} \mathop{}\!\mathrm{d}u\ \| \nabla W_2 \|_{L^{\infty}} \| \rho_{W_2}-\rho_{W_1} \|_{L^{2}_TL^2\cap H^1_TH^{-2}} \\
    &\lesssim_{d,T,M} (\| \rho_{W_1} \|_{L^{\infty}_{T}H^{d}}  + \| \rho_{W_2} \|_{L^{\infty}_{T}H^{d}}) \| \rho_{W_2}-\rho_{W_1} \|_{L^{2}_TL^2\cap H^1_TH^{-2}} \\
    &\leq C \| \rho_{W_2}-\rho_{W_1} \|_{L^{2}_TL^2\cap H^1_TH^{-2}}
\end{align*} 
for some constant \(C=C(d,T,M,\left\| \phi \right\|_{H^{d}})>0\). This proves the desired upper bound for the right-hand side of \eqref{eq:pre-stab-estimate-projected}.

Let us now prove a lower-bound for the left-hand side of \eqref{eq:pre-stab-estimate-projected}. To simplify the notation, let us henceforth use \(H \coloneqq W_2 - W_1\), and \(W \coloneqq W_1\). We are thus looking for a lower bound for \(\left\| \nabla \cdot (\rho_W \nabla H \ast \rho_W)  \right\|_{L^{2}_TH^{-2}}\). From \cite[Theorem 5]{mckv}, the densities \(\rho_W\) are uniformly bounded away from \(0\): \(\inf _{x \in \mathbb{T}^{d}, t\in [0,T]}\rho_W(t,x) \geq \rho_{\mathrm{min}} > 0\) and thus, with \(t \in [0,T]\) fixed, the operator \(u \mapsto -\nabla \cdot \left( \rho_W(t) \nabla u \right) \) is uniformly elliptic. By standard elliptic theory (see e.g. \cite[Chapter 6]{evans2010partial}) we know that for any \(\psi \in \dot{L}^{2}\) there is a unique \(u \in \dot{H}^{2}\) such that
\begin{equation}\label{eq:reg-estimate}
    - \nabla \cdot \left( \rho_W(t) \nabla u \right) = \psi, \quad \left\| u \right\|_{H^{2}} \leq C_{\mathrm{reg}} \left\| \psi  \right\|_{L^{2}} 
\end{equation} 
where \(C_{\mathrm{reg}}>0\) depends only on \(d, \rho_{\mathrm{min}}\) and an upper bound on \(\sup_{0 \leq t \leq T} \| \nabla \rho_W(t) \|_{L^{\infty}}\) (arguing as in the proof of Proposition A.5.3 in \cite{nickl2023bayesian}, using also the Poincar\'e inequality on the torus).  Let now \(\psi \coloneqq \rho_W(t) \ast H \in \dot{L}^{2}\) and \(u \in \dot{H}^{2}\) be the corresponding solution to the elliptic PDE above. Using integration by parts and the regularity estimate \eqref{eq:reg-estimate}, we find
\begin{align*}
    \left\| \nabla \cdot \left( \rho_W(t) \nabla \psi \right)  \right\|_{H^{-2}} = \sup_{\left\| v \right\|_{H^{2}}\leq 1} \left\langle \nabla \cdot \left( \rho_W(t) \nabla \psi \right), v \right\rangle _{L^{2}}&\geq \left\langle \nabla \cdot \left( \rho_W(t) \nabla \psi \right), -\frac{u}{\left\| u \right\|_{H^{2}}} \right\rangle _{L^{2}}\\
    &= \left\langle \psi , -\frac{\nabla \cdot \left( \rho_W(t) \nabla u\right)}{\left\| u \right\|_{H^{2}}} \right\rangle _{L^{2}}\\
    &= \frac{\left\| \psi  \right\|_{L^{2}}^{2}}{\left\| u \right\|_{H^{2}}} \geq \frac{1}{C_{\mathrm{reg}}} \left\| \psi \right\|_{L^{2}}.
\end{align*} 
Using \cite[Theorem 5]{mckv} and the Sobolev embeddings we obtain that \(C_{\mathrm{reg}}\) can be chosen to depend only on \(d,T,M,\phi_{\mathrm{min}},\| \phi \|_{H^{d}}\). Integrating the estimate over \([0,T]\) yields
\begin{equation}\label{eq:pre-stab-estimate-2}
\left\| \nabla \cdot (\rho_W \nabla H \ast \rho_W) \right\|_{L^{2}_T H^{-2}} \geq C_{\mathrm{reg}}^{-1} \left\| H \ast \rho_W\right\|_{L^{2}_{T}L^{2}}.
\end{equation} 
We now repeat the deconvolution argument of \cite[Theorem 2]{mckv} for the reader's convenience. We use that by virtue of \cite[Theorem 5]{mckv}, \(t \mapsto \rho_W(t)\) is a Lipschitz map from \([0,T] \to L^{1}\) with Lipschitz constant \(C= C \left( d, T, M, \left\| \phi \right\|_{H^{1}}\right) \) (since \(\sup_t \| \rho'_W(t) \|_{L^{2}}\leq C\)). Moreover, the \(L^{1}\)-norm uniformly bounds the Fourier coefficients so that we obtain, for all \(0 \leq t \leq T\),
\[
| \widehat{\rho}_W(t,k) - \widehat{\phi}_k | \leq \| \rho_W(t) - \rho_W(0) \|_{L^{1}} \leq C t 
\] 
and therefore, using equation \eqref{eq:fourier-decay-phi}, we deduce
\[
\left| \widehat{\rho}_W(t,k) \right| \geq c \left| k \right|^{-\zeta} - Ct.
\]
Hence, for \(t\) small enough, we obtain a lower bound on the Fourier coefficients of \(\rho_W(t)\)
\[
\inf_{\left| k \right|\leq K } \left| \widehat{\rho}_W(t,k) \right| \geq \frac{c_{\ast}}{2} K^{-\zeta}, \quad \text{for any }t \leq \min\big(T,\frac{c_{\ast}}{2C}K^{-\zeta}\big) \eqqcolon t_0.
\]
This means that for all \(t \leq t_0\),
\[
    \left\| H \ast \rho_W(t) \right\|_{L^{2}}^{2} \geq \sum_{\left| k \right|\leq K} | \widehat{h}_{k}  | ^{2} \left| \widehat{\rho}_{W}(t,k) \right|^{2} \geq \frac{c_{\ast}^{2}}{4} K^{-2 \zeta} \left\| W_2 - W_{1,K} \right\|_{L^{2}}^{2}.
\] 
Since the Lipschitz constant \(C\) grows with \(T\), we can always take it large enough (and independent of \(K\), since \(K \geq 1\)) so that \(t_0 = \frac{c_{\ast}}{2} K^{-\zeta}\). Therefore, 
\[
    \left\| H \ast \rho_W \right\|_{L^{2}_{T}L^{2}}^{2} \geq \int_{0}^{t_{0}} \left\| H \ast \rho_W(t) \right\|_{L^{2}}^{2} \mathop{}\!\mathrm{d}t\geq t_0 \frac{c_{\ast}^{2}}{4} K^{-2 \zeta}  \left\| W_2 - W_{1,K} \right\|_{L^{2}}^{2} \geq C K^{-3\zeta}  \left\| W_2 - W_{1,K} \right\|_{L^{2}}^{2}
\] where \(C=C(d,T,M,c_{\ast},\phi_{\mathrm{min}},\| \phi \|_{H^{d}}) >0\). Pugging this estimate back into \eqref{eq:pre-stab-estimate-2} yields the desired lower bound on the left-hand side of \eqref{eq:pre-stab-estimate-projected}. Combining it with the previously derived upper bound yields 
\[
    \left\| W_2-W_{1,K} \right\|_{L^{2}} \leq  S K^{3\zeta/2} \left(\left\| \rho_{W_2} - \rho_{W_1} \right\|_{L^{2}([0,T];L^{2})} + \left\| \rho_{W_2} - \rho_{W_1} \right\|_{H^{1}([0,T];H^{-2})} \right).
\] 
Noting $\|W_2-W_1\|_{L^2} \le \|W_1-W_{1,K}\|_{L^2} + \|W_2 - W_{1,K}\|_{L^2}$ completes the proof.
\end{proof}

\begin{remark}\label{remark:stab-estimate-linear-PDE}\normalfont [A pseudo-linearisation identity for the McKean--Vlasov equation.]  The preceding proof (via \eqref{eq:pre-stab-estimate}) also implies that \(v \coloneqq \rho_{W_2}-\rho_{W_1}\) is the solution to the linear PDE 
\[
    \begin{cases}
     \partial_t v - \Delta v - \nabla \cdot (v \nabla W_2 \ast \overline{\rho}) - \nabla \cdot (\overline{\rho}\ \! \nabla W_2 \ast v) = \nabla \cdot (\rho_{W_1} \nabla (W_2 - W_1) \ast \rho_{W_1}) \\
     v(0) = 0,\\
\end{cases} 
\] 
with \(\bar{\rho} \coloneqq (\rho_{W_1}+\rho_{W_2})/2\). From this one can directly derive forward estimates as in Theorem \ref{thm:lipschitz-continuity-forward-map} to follow, using the regularity theory in Appendix \ref{appendix:parabolic-regularity}.
\end{remark}

\subsubsection{Linearisation and higher derivatives}

Using the properties of \(f\) established in Propositions \ref{prop:derivatives-f} and \ref{prop:D2f-homeomorphism}, we can now apply Theorem \ref{thm:abstract-derivative} to obtain derivatives of the forward map \(W \mapsto \rho_W\) arising from the McKean--Vlasov equation (\ref{eq:mckv}).

\begin{theorem}[Linearisation of the McKean--Vlasov Equation]\label{thm:linearisation-mckean-vlasov}
    Let \(\phi\) satisfy Assumption \ref{assumption:regulatity-mckv} for some \(\beta \geq 3 + d\). For any \(W \in \dot{W}^{2,\infty}\), let \(\rho_W\) denote the unique solution to the PDE \eqref{eq:mckv} with initial condition $\phi$. Then, the forward map 
    \[
    \mathcal{G} : \dot{W}^{2,\infty} \to L^{2}([0,T];H^{\beta+1}) \cap H^{1}([0,T];H^{\beta-1})
    \] \(\) defined by \(\mathcal{G}(W) = \rho_W\) is infinitely differentiable in the Fréchet sense. 
    
    For any \(W,H \in \dot{W}^{2,\infty}\), its first Fréchet derivative \(D\rho_W[H]\) is the unique \(v \in L^{2}([0,T];H^{\beta+1}) \cap H^{1}([0,T];H^{\beta-1})\) solving the linear parabolic PDE 
    \begin{equation}\label{eq:mckv-first-derivative}
    \begin{aligned}
        \left( \partial _{t} - \mathcal{L}_W  \right)  v &= \nabla \cdot (\rho_W \nabla H \ast \rho_W) \\
        v(0) &= 0
    \end{aligned}
    \end{equation}
    where \(\mathcal{L}_W= \Delta [\cdot ] + \nabla \cdot ([\cdot ] \nabla W \ast \rho_W) + \nabla \cdot (\rho_W \nabla W \ast [\cdot ])\).

    For any \(W,H_1, H_2 \in \dot{W}^{2,\infty}\), its second Fréchet derivative \(D^{2}\rho_W [H_1,H_2]\) is the unique \break \(v \in L^{2}([0,T];H^{\beta+1}) \cap H^{1}([0,T];H^{\beta-1})\) solving the linear parabolic PDE
    \begin{equation}\label{eq:mckv-second-derivative}
    \begin{aligned}
        \left( \partial _{t} - \mathcal{L}_W  \right)  v &= \nabla \cdot (D \rho_W [H_2]\nabla H_1 \ast \rho_W)+\nabla \cdot (\rho_W \nabla H_1 \ast D \rho_W [H_2])\\
        &~~~+ \nabla \cdot (D \rho_W [H_1]\nabla H_2 \ast \rho_W)+\nabla \cdot (\rho_W \nabla H_2 \ast D \rho_W [H_1]) \\
        &~~~+ \nabla \cdot (D\rho_W [H_1] \nabla W \ast D \rho_W [H_2]) + \nabla \cdot (D\rho_W [H_2] \nabla W \ast D \rho_W [H_1]) \\
        v(0)&= 0.
    \end{aligned}
    \end{equation} 

\end{theorem}

\begin{proof}[Proof of Theorem \ref{thm:linearisation-mckean-vlasov}]
    We apply Theorem \ref{thm:abstract-derivative} with  \(E =W^{2,\infty}_{0}\), \(F = L^{2}([0,T];H^{\beta+1}) \cap H^{1}([0,T];H^{\beta-1})\), \(G = L^{2}([0,T];H^{\beta-1}) \times H^{\beta}\), \(U=E\times F\), \(u=\mathcal{G}\), \(V = E\). Propositions \ref{prop:derivatives-f} and \ref{prop:D2f-homeomorphism} ensure the conditions on \(f\) are satisfied. From \cite{chazelle_noisy_consensus}, \cite[Theorem 3.4]{Gvalani_thesis}, we know that for each \(W\), the solution \(\rho_W\) is unique in \(F\). Therefore, all the conditions of Theorem \ref{thm:abstract-derivative} are satisfied.

    The PDEs for the first and second derivatives of \(\mathcal{G}\) are given by plugging the expressions for the derivatives of \(f\) from Proposition \ref{prop:derivatives-f} into equation \eqref{eq:ift-second-derivative}. 
\end{proof}

The forward Lipschitz estimate below is a direct by-product of our techniques. It gives a new proof of \cite[Lemma 3. (ii)]{mckv} and produces a slightly better bound (with \(H^{-(\beta+1)}\)-norms on the right-hand side instead of \(H^{-\beta}\)).

\begin{theorem}[Forward estimate]\label{thm:lipschitz-continuity-forward-map} Let \(\phi\) satisfy Assumption \ref{assumption:regulatity-mckv} for some \(\beta \geq 3 + d\). For any \(W \in \dot{W}^{2,\infty}\), let \(\rho_W\) denote the unique solution to the PDE \eqref{eq:mckv} with initial condition $\phi$. Let \(M>0\) and \(W_1, W_2 \in \dot{W}^{2,\infty}\) be such that \mbox{\(\| W_1 \|_{W^{2,\infty}}+\| W_1 \|_{W^{2,\infty}} \leq M\)}. Then, there exists a constant \(L>0\) depending on \(d,T,M,\left\| \phi \right\|_{H^{d+1}}\) such that
    \[
    \left\| \rho_{W_2} - \rho_{W_1} \right\|_{L^{2}([0,T];L^{2})} \leq L \left\| W_2 - W_1 \right\|_{H^{-(\beta +1)}}.
    \]   
\end{theorem}

\begin{proof}[Proof of Theorem \ref{thm:lipschitz-continuity-forward-map}] 
    By the fundamental theorem of calculus in Banach spaces (e.g. \cite[Proposition A.2.3]{LiuRockner_introSPDEs}) and Theorem \ref{thm:linearisation-mckean-vlasov},
    \[
    \rho_{W_2} - \rho_{W_1} = \int_{0}^{1} D\rho_{W_1 + (W_2-W_1)u} [W_2 - W_1] \mathop{}\!\mathrm{d}u
    \] with convergence of the integral in \(L^{2}([0,T];H^{\beta+1}) \cap H^{1}([0,T];H^{\beta-1})\), and therefore in \(L^{2}([0,T];L^{2})\) as well. Taking norms and using the triangle inequality yields
    \begin{equation}\label{eq:forward-estimate-1}
        \left\| \rho_{W_2} - \rho_{W_1} \right\|_{L^{2}([0,T];L^{2})} \leq \int_{0}^{1} \left\| D\rho_{W_u} [H] \right\|_{L^{2}([0,T];L^{2})} \mathop{}\!\mathrm{d}u
    \end{equation} 
    with \(W_u \coloneqq W_1 + (W_2-W_1)u\) and \(H \coloneqq W_2 - W_1\).  We use that \(D\rho_W [H]\) solves \eqref{eq:mckv-first-derivative} together with the regularity estimate from Theorem \ref{thm:linPDE-weak-norms} with \(s=2\) to obtain, for some \(L>0\) large enough,
    \begin{align*}
        \left\| D\rho_{W_u} [H] \right\|_{L^{2}([0,T];L^{2})} &\lesssim  \left\| \nabla \cdot (\rho_{W_u} \nabla H \ast \rho_{W_u}) \right\|_{L^{2}([0,T];H^{-2})} \\
        &\lesssim  \sup_{0\leq t \leq T}\left\| \rho_{W_u} \right\|_{C^{1}} \left\| \rho_{W_u} \right\|_{L^{2}_{T}H^{\beta +1}} \left\| H \right\|_{H^{-(\beta +1)}} \\
        &\leq L \left\| H \right\|_{H^{-(\beta +1)}}
    \end{align*} 
    where we have used Lemma \ref{lem:trilin-op-T-bound} equation \eqref{eq:trilin-negative} with \(\nu = \beta + 1\) in the second inequality and the Sobolev embeddings and \cite[Theorem 5]{mckv} in the last inequality. The constant \(L\) can be chosen large enough to only depend on \(d,T,M,\left\| \phi \right\|_{H^{d}}\) (and not on \(s\)) since \(\| W_u \|_{W^{2,\infty}} \leq 2M\). Therefore, plugging this estimate back into \eqref{eq:forward-estimate-1} concludes the proof.
\end{proof}

Higher derivatives of \(W \mapsto \rho_W\) can also be bounded in appropriate norms, as the lemma below shows. While we will use this only for $k=1,2,3$, we include the general statement.

\begin{lemma}\label{lem:bound-on-derivatives-of-forward-map}
    Let \(\phi\) satisfy Assumption \ref{assumption:regulatity-mckv} for some \(\beta \geq 3 + d\). For any \(W \in \dot{W}^{2,\infty}\), let \(\rho_W\) denote the unique solution to the PDE \eqref{eq:mckv} with initial condition $\phi$. Let \(k \geq 1\) a natural integer, \(M>0\), \(\frac{d}{2} < \gamma \leq \beta\), and \(W \in \dot{W}^{2,\infty}\) such that \(\| W \|_{W^{2,\infty}} \leq M\). There exists \(C=C(k,d,\gamma,T,M, \left\| \phi \right\|_{H^{\gamma }})>0\) such that for all \(h_1, \dots, h_k \in E\),
    \[
    \left\| D^{k}\rho_W [h_1, \dots, h_k] \right\|_{L^{2}([0,T];H^{\gamma+1})\cap H^{1}([0,T],H^{\gamma-1})} \leq C \left\| h_1 \right\|_{L^{2}} \cdots \left\| h_k \right\|_{L^{2}}.
    \]   
\end{lemma}
\begin{proof}
    See Appendix \ref{appendix:further-technical-results}.
\end{proof}

The bound on the right-hand side of Lemma \ref{lem:bound-on-derivatives-of-forward-map} above can be upgraded to hold with weaker $H^{-a_k}$ norms for suitable $a_k>0$, but this will not be needed in the sequel.
%\ac{ \begin{remark}\label{rem:better-lemma} \normalfont Lemma \ref{lem:bound-on-derivatives-of-forward-map} can be upgraded to 
%    \[   \left\| D^{k}\rho_W [h_1, \dots, h_k] \right\|_{L^{2}([0,T];H^{\gamma+1})\cap H^{1}([0,T],H^{\gamma-1})} \leq C \left\| h_1 \right\|_{H^{-a_1}} \cdots \left\| h_k \right\|_{H^{-a_k}}\]
%    for any reals \(a_1, \dots, a_k\) such that \(a_1 + \cdots + a_k \leq \beta - \gamma\). Indeed, one can change the proof by using Lemma \ref{lem:trilin-op-T-bound} equation \eqref{eq:trilin-algebra} with \(\nu=-a_j\) instead. Since this optimisation does not improve our results, we omit the details here.
%\end{remark}} 
We end this section by verifying the gradient stability condition (\ref{gradstab}).

\begin{theorem}[Gradient stability]\label{prop:gradient-stability}
    Let the initial condition \(\phi\) satisfy Assumptions \ref{assumption:regulatity-mckv} and  \ref{assumption:phi} for some \(\beta \ge 3+d\) and \(\zeta > \beta + d/2\). 
Let \(W \in \dot{W}^{2,\infty}\) be such that \(\left\| W \right\|_{W^{2,\infty}} < M\) and denote by \(\rho_W\) the unique solution to the PDE \eqref{eq:mckv} with initial condition $\phi$. Consider the space \(\dot{E}_K\) of mean-zero trigonometric polynomials of degree at most \(K \ge 1\) from \eqref{eq:def-E_K}. Then, there exists a constant \(C=C(d,T,M,\beta,c_{\ast},\phi_{\mathrm{min}},\left\| \phi \right\|_{H^{\beta}})>0\) (independent of \(K\)) such that for any \(H \in \dot{E}_K\),
    \[
    \left\| D \rho_W [H]\right\|_{L^{2}([0,T];L^{2})} \geq C K^{-3\zeta} \left\| H \right\|_{L^{2}}.
    \]
\end{theorem}

\begin{proof}[Proof of Theorem \ref{prop:gradient-stability}] 
    We use that, by Theorem \ref{thm:linearisation-mckean-vlasov}, \(D \rho_W [H]\) solves the linear PDE \eqref{eq:mckv-first-derivative}, so
    \begin{equation}\label{eq:lin-solves-pde}
        \left( \partial _{t} - \mathcal{L}_W  \right)  D \rho_W [H] = \nabla \cdot (\rho_W \nabla H \ast \rho_W) \text{ in } L^{2}([0,T];H^{\beta-1})
    \end{equation} 
    where we recall that \(\mathcal{L}_W= \Delta [\cdot ] + \nabla \cdot ([\cdot ] \nabla W \ast \rho_W) + \nabla \cdot (\rho_W \nabla W \ast [\cdot ])\).
    As we have shown in the proof of Theorem \ref{thm:mckv-stability-estimate} (see \eqref{eq:pre-stab-estimate-2} and subsequent calculations), the forcing term on the right-hand side  is lower bounded as: 
    \[ \left\| \nabla \cdot (\rho_W \nabla H \ast \rho_W) \right\|_{L^{2}_TH^{-2}}^{2} \gtrsim \left\| H \ast \rho_W \right\|_{L^{2}_{T}L^{2}}^{2} \geq C K^{-3\zeta} \left\| H \right\|_{L^{2}}^{2},~~~H \in \dot{E}_K
    \] with \(C=C(d,T,M,c_{\ast},\phi_{\mathrm{min}},\| \phi \|_{H^{d}}) >0\). Therefore, it suffices to show that for all $H \in \dot{E}_K$
    \begin{equation}\label{eq:sts-gradient-stab}
    \|\left(\partial _{t} - \mathcal{L}_W  \right)  D \rho_W [H] \|_{L^{2}_TH^{-2}}^{2} \lesssim \| D \rho_W [H] \|_{L^{2}_{T}L^{2}} \left\| H \right\|_{L^{2}}. 
    \end{equation} 
    Using the definition of \(\mathcal{L}_W\), Lemma \ref{lem:trilin-op-T-bound} equation \eqref{eq:trilin-leibniz} with \(k=l=0\) and \(\nu=1\), the Sobolev embeddings and \cite[Theorem 5]{mckv}, and finally interpolation (e.g. \cite[Ch. 4, Prop. 2.1]{lions}), we see that for any sufficiently smooth \(u\),  
    \begin{align}
        \|\! \left( \partial _{t} - \mathcal{L}_W  \right)\!  u \|_{L^{2}_TH^{-2}}^{2}\! &\lesssim \| \partial_t u\|_{L^{2}_TH^{-2}}^{2} + \| \Delta u \|_{L^{2}_TH^{-2}}^{2} + \| \nabla \cdot (u \nabla W \ast \rho_W) \|_{L^{2}_TH^{-2}}^{2} + \| \nabla \cdot (\rho_W\nabla W \ast u) \|_{L^{2}_TH^{-2}}^{2} \nonumber \\
        &\lesssim \| u\|_{H^{1}_TH^{-2}}^{2} + \| u\|_{L^{2}_TL^{2}}^{2} (1 + \| \nabla W \|_{L^{\infty}}^{2} \| \rho_W \|_{L^{\infty}_TL^{\infty}}^{2}) \nonumber \\
        & \lesssim \| u\|_{H^{1}_TH^{-2}}^{2} + \| u\|_{L^{2}_TL^{2}}^{2} \nonumber \\
        & \leq C \| u\|_{L^{2}_TL^{2}} (\| u\|_{H^{2}_TH^{-2}} + \| u\|_{L^{2}_TL^{2}}) \label{eq:interpolation-gradient-stability}
    \end{align} 
    with \(C = C(d,T,M,\| \phi \|_{H^{d}})\). Using the above with \(u = D \rho_W[H]\), we deduce that proving \eqref{eq:sts-gradient-stab} reduces to showing 
    \[
    A + B \coloneqq \| D \rho_W[H] \|_{H^{2}_TH^{-2}} + \| D \rho_W[H]\|_{L^{2}_TL^{2}} \lesssim \left\| H \right\|_{L^{2}}.
    \] 
    The desired bound on \(B\) follows directly from an application of Lemma \ref{lem:bound-on-derivatives-of-forward-map} with \(k=1\) and \(\gamma=d\). It remains to bound \(A\). For this, recall again that \(D \rho_W[H]\) solves the PDE \eqref{eq:lin-solves-pde} so that the regularity estimate of Theorem \ref{thm:linPDE-time-regularity}, Lemma \ref{lem:trilin-op-T-bound} with \(k=1\) and \(\nu=0\), the Sobolev embeddings, and \cite[Theorem 5]{mckv} yield
    \begin{align*}
        A = \| D \rho_W [H] \|_{H^{2}_{T}H^{-2}} &\lesssim \left\| \nabla \cdot (\rho_W \nabla H \ast \rho_W) \right\|_{H^{1}_TL^{2}} \\
        &\lesssim \left\| \nabla \cdot (\rho_W \nabla H \ast \rho_W) \right\|_{L^2_TL^{2}} + \left\| \frac{\mathop{}\!\mathrm{d}}{\mathop{}\!\mathrm{d}t} \nabla \cdot (\rho_W \nabla H \ast \rho_W) \right\|_{L^{2}_TL^{2}}\\
        &\lesssim_{d,T} \left\| H \right\|_{L^{2}} \left\| \rho_W \right\|_{H^{1}_{T}C^{1}}\left\| \rho_W \right\|_{L^{\infty}_{T}H^{2}} \leq C \left\| H \right\|_{L^{2}}
    \end{align*} 
    with \(C=C(d,T,M,\beta,\left\| \phi  \right\|_{H^{\beta}})\). This concludes the proof of \eqref{eq:sts-gradient-stab}, and hence the proof of the proposition.
\end{proof}

    We have not attempted to optimise the gradient stability exponent in the preceding result (as it will not be relevant below). If necessary, using higher time-regularity estimates as in \cite[Sec. 7.1.3 Theorem 6]{evans2010partial}, one could obtain the slightly better \(\left\| D \rho_W [H] \right\|_{L^{2}_{T}L^{2}} \geq CK^{-3\zeta/(2\nu)} \left\| H \right\|_{L^{2}}\) for $1/2<\nu<1$.

\section{Polynomial-time posterior sampling via Langevin-type algorithms}

Inference algorithms for interacting particle systems have been widely studied recently, we refer to e.g., \cite{della2022nonparametric, amorino2024polynomial, pavliotis2025linearization, mckv} and references therein. The approach proposed in \cite{mckv} is to infer the interaction potential $W$ from noisy measurements of the solution to the McKean--Vlasov equation (\ref{eq:mckean-vlasov-intro}) using techniques from Bayesian non-linear inverse problems 
\cite{Stuart2010, MNPCPAM, nickl2023bayesian} applied to the forward map $\theta \mapsto \mathcal G(\theta)$ in (\ref{eq:energy}) arising from $W \mapsto \rho_W$ (so with \(\theta = W\) here). In \cite{mckv} it was proved that the posterior mean vector for $W$ arising from a Gaussian process prior provides a statistically consistent estimator for $W$ under appropriate hypotheses. In practice such estimators are computed by MCMC methods, and the purpose of this subsection is to deploy the techniques from Section \ref{sec:application-mckv} to show that the unadjusted discretised Langevin algorithm mentioned in the introduction can be used to approximate such posterior means with any prescribed accuracy after a number of iterations that scales polynomially in relevant parameters. That such a polynomial runtime algorithm exists is not a priori clear due to the lack of global log-concavity of the posterior Gibbs measure (\ref{eq:posterior-intro}) for the non-linear map $\theta \mapsto \mathcal G(\theta)$ -- we refer to \cite{BMNW23} for negative results on the scaling of MCMC algorithms in absence of strong hypotheses on $H$.
 
\subsection{Setting}

We begin by fixing the constants used throughout this section. Recall that $d \ge 1$ denotes the dimension of the torus $\mathbb{T}^d$ on which the interaction potential $W$ is defined. The parameter $\beta$ represents the Sobolev regularity of the initial condition $\phi$ as in Assumption \ref{assumption:regulatity-mckv}, while $\zeta$ characterizes the lower bound \eqref{eq:fourier-decay-phi} on the Fourier modes of $\phi$.

We introduce two additional real-valued constants: $\alpha$, representing the Sobolev regularity of the ground truth $W_0$ that generates the regression data (see \eqref{eq:data} below); and $w$, a scaling exponent that determines the radius $\mathbf{r} \simeq D^{-w}$ of the ball \(\mathcal{B}_\mathbf{r}\) defined in \eqref{eq:ball}. One should think of \(\mathcal{B}_{\mathbf{r}}\) as the local region around the ground truth \(W_0\) where the posterior distribution \eqref{eq:posterior-intro} will be shown to be log-concave with high frequentist probability. These parameters are assumed to satisfy the following (as is easily seen, non-empty) conditions:

\begin{equation}\label{eq:constants}
\begin{cases}
\beta \geq 4 +d \quad \text{and is an even integer}, \\[5pt]
\alpha > 12 \beta + 6d - 1, \\[5pt]
\beta+\frac{d}{2} < \zeta < \frac{\alpha+1}{12}, \\[5pt]
\frac{6 \zeta}{d} < w < \frac{1}{d} \min \left( \frac{(\alpha+1)(\beta-2)}{\beta} - \frac{3\zeta}{2}, \alpha+ 1 - 6\zeta \right).
\end{cases}    
\end{equation}
Recall that for any such choices of \(\beta\) and \(\zeta\), one can construct an initial condition \(\phi\) satisfying Assumptions \ref{assumption:regulatity-mckv} and \ref{assumption:phi} as discussed below Assumption \ref{assumption:phi}.

\subsubsection{Approximation spaces}\label{subsec:approximation-spaces}

We first construct an orthonormal basis of \(L^{2}(\mathbb{T}^{d})\) by tensorizing the one-dimensional real Fourier basis. Specifically, for \(k =(k_1, \dots, k_d) \in \mathbb Z^{d}\), \(x=(x_1,\dots, x_d) \in \mathbb{T}^{d}\), and \(m \in \mathbb{Z}\), \(y \in \mathbb{T}\), define
\begin{equation}\label{eq:trig-basis}
    \tau_k(x) \coloneqq \prod_{j=1}^d T_{k_j}(x_j), \qquad
T_m(y) \coloneqq
\begin{cases}
     \sqrt{2} \cos(2\pi m y) & \text{for } m > 0, \\
     1 & \text{for } m = 0, \\
     \sqrt{2} \sin(2\pi m y) & \text{for } m < 0.
\end{cases}
\end{equation}
For \(K\geq1\), we let \(\dot{E}_K\) be the space of \textit{mean-zero} trigonometric polynomials of degree at most \(K\), which we will use as our discretisation space for \(W\). That is,
\begin{equation}\label{eq:def-E_K}
    \dot{E}_{K} \coloneqq \mathrm{span}_{\mathbb R} \left\{ \tau_k : k \in \mathbb{Z}^{d}, k \neq 0, \left| k \right| \leq K\right\} = \mathfrak{R} \left( \mathrm{span}_{\mathbb C} \left\{ e_k : k \in \mathbb{Z}^{d}, k \neq 0, \left| k \right| \leq K \right\} \right) 
\end{equation}  
where \(e_k : x \mapsto e^{2i\pi k \cdot x}\), and where $\mathfrak R$ denotes real parts. The second equality is easily proved using standard product-to-sum identities for trigonometric functions. The dimension of the discretisation space \(\dot{E}_K\) will be denoted by \(D\). The spaces \(\dot{E}_K\) and \(\mathbb{R}^{D}\) being isomorphic, we will tacitly identify them in what follows.

Note that with this choice of basis, the truncation happens at the same level as the truncation in the basis \((e_k)_k\) used in the main reference \cite{mckv}, in the sense that for any (real-valued) \(f \in L^{2}\) and any truncation level \(K\),   
\[
\sum_{|k| \leq K} \langle f, e_k \rangle_{L^{2}} \ \!e_k = \sum_{|k| \leq K} \langle f, \tau_k \rangle_{L^{2}}\ \!\tau_k.
\]
This implies that when \(f \in \dot{E}_K\), its Sobolev norms are controlled by its \(L^2\)-norm times a penalty depending on \(K\), a fact which we shall use repeatedly in the sequel. More precisely, for any \(s \geq 0\) and \(f \in \dot{E}_K\) we have that,
\begin{equation} \label{terminator}
\| f \|_{H^{s}}^2 \lesssim \sum_{| k | \leq K} | k |^{2s} | \left\langle f,e_k \right\rangle  |^{2} \leq K^{2s} \| f \|^{2}_{L^{2}}.
\end{equation}

\begin{remark}\label{remark:equivalence-D-K^d} \normalfont
    Recall that \(D = \mathrm{dim}(\dot{E}_K) = \#\left\{ k \in \mathbb{Z}^d : \left| k \right| \leq K \right\} - 1\). By covering the unit ball in \(\mathbb{R}^{d}\) with cubes of side length \(1\), we see that when \(K \geq \sqrt{d}/2\),
\[
\omega_d (K-\sqrt{d}/2)^{d} \leq D(K) + 1 \leq \omega_d (K+\sqrt{d}/2)^{d}
\] where \(\omega_d\) is the volume of the unit ball. From this we immediately conclude that \(D(K) \sim \omega_d K^{d}\) as \(K \to \infty\), and that
\begin{equation}\label{eq:equivalence-D-K^d}
    D(K) \simeq_d K^{d} \text{ for any } K \geq 1.
\end{equation}
We will use this equivalence throughout to switch between the conditions of \cite[Chapter 5]{nickl2023bayesian} stated in terms of \(D\) and conditions stated in terms of \(K^{d}\).
\end{remark}

\subsubsection{Derivatives of the likelihood function}

Let \(\phi\) satisfy Assumption \ref{assumption:regulatity-mckv} for some \(\beta \geq 3 + d\). For any \(W \in \dot{W}^{2,\infty}\), let \(\rho_W\) denote the unique solution to the PDE \eqref{eq:mckv} with initial condition $\phi$. Define the log-likelihood per single observation $(Y,t,X) \in \mathbb R \times [0,T] \times \mathbb T^d$ as
\begin{equation}\label{eq:log-likelihood}
   \ell (W) \coloneqq - \frac{1}{2}| Y - \rho_W(t,X) | ^{2}, \quad W \in \dot{E}_K,
\end{equation}
with gradient given by 
\[
\nabla \ell(W) \coloneqq \lim_{h \to 0} \left( \frac{\ell(W+h \tau_k) - \ell(W)}{h}\right)_{\substack{\left| k \right| \leq K \\ k \neq 0}} \eqqcolon \left( \frac{\partial \ell}{\partial{W_{k}}}(W) \right)_{\substack{\left| k \right| \leq K \\ k \neq 0}}
\] 
which we view as a vector in \(\mathbb{R}^{D}\). The Hessian \(\nabla ^{2} \ell \) is defined analogously as the \(D \times D\) matrix of second partial derivatives. The definition of these derivatives will be used for any function from \(\dot{E}_K\) to \(\mathbb{R}\).

By \cite[Theorem 5]{mckv}, Theorem \ref{thm:linearisation-mckean-vlasov} and the Sobolev imbedding, the solution \(\rho_W\) as well as its higher Fréchet derivatives \(D^{k}\rho_W [H_1,\dots,H_k], k=1,2,\) all have representatives that are jointly continuous in \((t,x)\) so that point evaluation is well-defined. Therefore, for \(\nabla \ell \) and \(\nabla ^{2} \ell \)  to exist, it suffices that the map 
\begin{equation}\label{eq:Gtx}
\mathcal{G}^{t,x} : \dot{E}_{K} \to \mathbb R, \quad W \mapsto \rho_W(t,x)    
\end{equation} 
admits a gradient and a Hessian, which is the purpose of the next lemma.

\begin{lemma}\label{lem:discrete-derivatives-forward-map}
    Let \(\phi\) satisfy Assumption \ref{assumption:regulatity-mckv} for some \(\beta \geq 3 + d\). For any \(W \in \dot{W}^{2,\infty}\), let \(\rho_W\) denote the unique solution to the PDE \eqref{eq:mckv} with initial condition $\phi$ and let \(D^{k} \rho_W\) be the \(k\)-th Fréchet derivative of \(W \mapsto \rho_W\) given by Theorem \ref{thm:linearisation-mckean-vlasov}. For any \((t,x) \in [0,T] \times \mathbb{T}^{d}\) and \(W \in \dot{W}^{2,\infty}\), the gradient \(\nabla \mathcal{G}^{t,x}(W)\) and Hessian \(\nabla ^{2} \mathcal{G}^{t,x}(W)\) of \(\mathcal{G}^{t,x}\) in \eqref{eq:Gtx} exist and are given by
\begin{equation}\label{eq:discrete-to-continuous}
h_1^{T} \nabla \mathcal{G}^{t,x} (W) =  D \rho_W [h_1] (t,x), \quad  h_1^{T} \nabla ^{2} \mathcal{G}^{t,x} (W) h_2 =  D^{2} \rho_W [h_1,h_2] (t,x) 
\end{equation} 
for any \(h_1,h_2 \in \mathbb{R}^{D}\).
\end{lemma}

\begin{proof}[Proof of Lemma \ref{lem:discrete-derivatives-forward-map}]
    By linearity of \(D \rho_W\) and \(D^{2} \rho_W\), it suffices to show that the limit of \break \(\left( \mathcal{G}^{t,x}(W+h \tau_k) - \mathcal{G}^{t,x}(W) \right)/h\) exists as \(h\to 0\) and is equal to \(D\rho_W[\tau_k](t,x)\). Recall that by \eqref{eq:trace-theorem-time} and the Sobolev embedding, \(L^{2}_{T}H^{\beta+1}\cap H^{1}_{T}H^{\beta-1} \hookrightarrow L^{\infty}_{T}L^{\infty}\). Therefore,  
    \begin{align*}
        \left|  \mathcal{G}^{t,x}(W+h\tau_k) - \mathcal{G}^{t,x}(W) - h D \rho_W [\tau_k](t,x) \right| &\leq \left\| \rho_{W+h\tau_k} - \rho_W - D \rho_W [h\tau_k] \right\|_{L^{\infty}([0,T],L^{\infty})}\\
        &\lesssim \left\| \rho_{W+h\tau_k} - \rho_W - D \rho_W [h\tau_k] \right\|_{L^{2}_{T}H^{\beta+1}\cap H^{1}_{T}H^{\beta-1}} \\
        &= o(\left\| h\tau_k \right\|_{W^{2,\infty}})=o(h)
    \end{align*} 
    where in the last line we have used the fact that \(D \rho_W\) is the Fréchet derivative of \(\mathcal{G} : W^{2,\infty} \to L^{2}([0,T],H^{\beta +1})\cap H^{1}([0,T],H^{\beta -1})\) from Theorem \ref{thm:linearisation-mckean-vlasov}. Dividing by \(\left| h \right|\)  concludes the proof. 
\end{proof}

\subsection{Local properties of the likelihood function}

\subsubsection{Local average curvature}\label{subsec:local-average-curvature}

Given a ground truth parameter \(W_0 \in \dot{H}^{\alpha+1}\), define \(W_{0,K}\) to be its projection onto \(\dot{E}_K\) from \eqref{eq:def-E_K} and, for \(\mathbf{r} >0\) to be chosen, the ball 
\begin{equation}\label{eq:ball}
    \mathcal{B}_{\mathbf{r}}= \{W \in \dot{E}_K : \| W-W_{0,K} \|_{L^{2}} \leq \mathbf{r} \}.
\end{equation} 
If \(\| W_0 \|_{H^{\alpha+1}} \leq M\), then \(\mathcal{B}_{\mathbf{r}}\) with \(\mathbf{r}=rD^{-w}\) is contained in a \(W^{2,\infty}\)-ball of radius \(\bar{M}(d,r,M)\) uniform in \(K\) and \(W_0\): indeed, if \(W \in \mathcal{B}_{\mathbf{r}}\) then using the Sobolev embeddings, (\ref{terminator}) and \eqref{eq:equivalence-D-K^d},
\begin{equation}\label{eq:Br-contained-in-W2infty-ball}
    \| W \|_{W^{2,\infty}} \lesssim_d \| W_{0,K} \|_{H^{\alpha+1}} + \| W - W_{0,K} \|_{H^{2+d}} \lesssim_d M + rK^{2+d-wd} \leq C_d (M + r) \eqqcolon \bar{M}
\end{equation} 
where we used that \(wd > 6 \zeta > \beta > d + 2\) by the choice of constants in \eqref{eq:constants}. The following theorem shows that the negative log-likelihood function is convex \textit{on average} over the local high-dimensional region \(\mathcal{B}_{\mathbf{r}}\) -- in particular it verifies the key curvature condition given in \cite[Equation (43)]{NicklWang2024_polynomialTimeLangevin}. Note that the stability exponent naturally depends on the amount of `ill-posedness' of the underlying deconvonlution problem via the constant $\zeta$ from Assumption \ref{assumption:phi}. 
\begin{theorem}[Local average curvature]\label{thm:local-average-curvature}
    Choose \(\alpha, \beta, \zeta\) and \(w\) as in \eqref{eq:constants}. Let the initial condition \(\phi\) satisfy Assumptions \ref{assumption:regulatity-mckv} and \ref{assumption:phi} for this choice of \(\beta\) and \(\zeta\). For any \(W \in \dot{W}^{2,\infty}\), let \(\rho_W\) denote the unique solution to the PDE \eqref{eq:mckv} with initial condition $\phi$.
    Let \(M>0\) and \(W_0 \in \dot{H}^{\alpha+1}\) such that \(\left\| W_0 \right\|_{W^{2,\infty}} \leq M\). Then, there exist \(0<r_{\mathrm{max}}\leq 1\) and \(c_0>0\) depending on \(d,T,M,\beta,c_{\ast},\phi_{\mathrm{min}},\| \phi \|_{H^{\beta}}\) such that for any \(0 < r \leq r_{\mathrm{max}}\), \(\mathbf{r} = rD^{-w}\), \mbox{\(\| \rho_{W_0} - \rho_{W_{0,K}} \|_{L^{2}([0,T];L^{2})} \leq \mathbf{r}\)} we have
    \[
    \inf_{W\in \mathcal{B}_{\mathbf{r}}} \lambda_{\min} \left( \mathbb{E}_{W_0} [- \nabla ^{2}\ell (W)] \right) \geq c_{0} D^{-6\zeta/d},
    \]
    where \(\mathcal{B}_{\mathbf{r}}\) is as in \eqref{eq:ball} for $K \ge 1$, \(\lambda_{\min}\) denotes the smallest eigenvalue of the matrix inside the parentheses, and the expectation is taken under the law \(P^1_{W_0}\) from \eqref{eq:data}.
\end{theorem}

\begin{proof}[Proof of Theorem \ref{thm:local-average-curvature}]
Let \(W \in \mathcal{B}_{\mathbf{r}}\) and recall from \eqref{eq:Br-contained-in-W2infty-ball} that \(\| W \|_{W^{2,\infty}} \leq \bar{M}\) for some \(\bar{M} = \bar{M}(M,d) > 0\) since \(r \leq r_{\max} \leq 1\).
    From the definition of \(\ell \) in \eqref{eq:log-likelihood} we have that for any \(W \in \dot{E}_K\),    
\begin{align*}
    \nabla \ell (W) &= (Y - \rho_W(t,X)) \nabla \mathcal{G}^{t,X}(W), \\
    \nabla ^{2} \ell (W) &= - \nabla \mathcal{G}^{t,X}(W) (\nabla \mathcal{G}^{t,X}(W))^{T} + (Y - \rho_W(t,X)) \nabla ^{2} \mathcal{G}^{t,X}(W).
\end{align*} 
Therefore, using \eqref{eq:discrete-to-continuous} and taking expectation under \(P_{W_0}^1\) from \eqref{eq:data},
\[
h^{T}\mathbb{E}_{W_0}[- \nabla ^{2}\ell (W)]h = \frac{1}{T} \left\| D \rho_W [h] \right\|_{L^{2}_TL^{2}}^{2} - \frac{1}{T}\left\langle \rho_W-\rho_{W_0}, D^{2}\rho_W [h,h] \right\rangle _{L^{2}_TL^{2}}
\] for any \(h \in \mathbb{R}^{d}\). Now, by applying Theorem \ref{prop:gradient-stability} to the first term and the Cauchy-Schwarz inequality to the second, we deduce
\[
h^{T}\mathbb{E}_{W_0}[- \nabla ^{2}\ell (W)]h \gtrsim_T CK^{-6\zeta} \| h \|_{L^{2}}^{2} - \| \rho_W - \rho_{W_0} \|_{L^{2}_{T}L^{2}} \left\| D^{2}\rho_W [h,h] \right\|_{L^{2}_TL^{2}}.
\]  
But by Theorem \ref{thm:lipschitz-continuity-forward-map}, the assumption on \(\| \rho_{W_0} - \rho_{W_{0,K}} \|_{L^{2}([0,T];L^{2})}\), (\ref{eq:equivalence-D-K^d}), and the fact that \(W \in \mathcal{B}_{\mathbf{r}}\), we obtain 
\[
\| \rho_W - \rho_{W_0} \|_{L^{2}_{T}L^{2}} \leq \| \rho_W - \rho_{W_{0,K}} \|_{L^{2}_{T}L^{2}} + \| \rho_{W_{0,K}} - \rho_{W_0} \|_{L^{2}_{T}L^{2}} \lesssim_d L\| W - W_{0,K} \|_{L^{2}} + r K^{-dw} \lesssim_d r (1 + L)K^{-dw}.
\] 
Moreover, by Lemma \ref{lem:bound-on-derivatives-of-forward-map} with \(k=2\) and \(\gamma=d\), we see that
\[
\left\| D^{2}\rho_W [h,h] \right\|_{L^{2}_TL^{2}} \leq C' \| h \|_{L^{2}}^{2}.
\] 
Putting these estimates together, we obtain
\[
h^{T}\mathbb{E}_{W_0}[- \nabla ^{2}\ell (W)]h \gtrsim_{d,T} (C K^{-6\zeta} - C'r(1 + L) K^{-wd}) \| h \|_{L^{2}}^{2}.
\] 
Using further that \(wd > 6\zeta\) it follows that there is \(0<r_{\max} \leq 1\) small enough and \(c_0>0\) depending on \(d,T,M,\beta,\| \phi \|_{H^{\beta}}\) such that for any \(0 < r \leq r_{\max}\),
\[
h^{T}\mathbb{E}_{W_0}[- \nabla ^{2}\ell (W)]h \geq c_0 K^{-6\zeta} \| h \|_{L^{2}}^{2}.
\] 
Finally, taking infima over all \(h \in \mathbb{R}^{d}\) with \(\| h \|_{L^{2}}=1\) and \(W \in \mathcal{B}_{\mathbf{r}}\), as well as applying (\ref{eq:equivalence-D-K^d}), concludes the proof.  
\end{proof}

\subsubsection{Local regularity}

For \(G : \dot{E}_K \to \mathbb{R}\), define the \(C^{2,1}(\mathcal{B})\)-norm as  
\begin{equation*}
        \left\| G \right\|_{C^{2,1}(\mathcal{B})} = \sup_{\substack{w,w' \in \mathcal{B} \\ w \neq w'}} \left\{ \left|G(w) \right|  + \left\| G(w) \right\|_{D} + \left\| \nabla ^{2}G(w) \right\|_{D} + \frac{\left\| \nabla ^{2} G(w) - \nabla ^{2} G(w') \right\|_{D}}{\left\| w-w' \right\|_{D} } \right\} 
\end{equation*} where \(\left\| \cdot  \right\|_{D}\) is the Euclidean norm on \(\mathbb{R}^{D}\) when applied to a vector, and the operator norm induced by the Euclidean norm when applied to a matrix.

\begin{proposition}[Local regularity]\label{prop:local-regularity}
    Choose \(\alpha, \beta, \zeta\) and \(w\) as in \eqref{eq:constants}. Let \(\phi\) satisfy Assumption \ref{assumption:regulatity-mckv} for this choice of \(\beta\). Let \(M>0, K \geq 1\) and \(W_0 \in \dot{H}^{\alpha+1}\) such that \(\left\| W_0 \right\|_{W^{2,\infty}} \leq M\). Then, there exists \(c_1(d,T,M,\| \phi \|_{H^{d}})>0\) such that, after choosing \(\mathbf{r} = rD^{-w}\) for the radius of the ball \(\mathcal{B}_{\mathbf{r}}\) in \eqref{eq:ball} with any \(0 < r \leq 1\), 
    \[
    \sup _{(t,x) \in [0,T]\times \mathbb{T}^{d}} \left\| \mathcal{G}^{t,x} \right\|_{C^{2,1}(\mathcal{B}_{\mathbf{r}})} \leq c_1.
    \]
\end{proposition}

\begin{proof}[Proof of Proposition \ref{prop:local-regularity}]
    We start by writing \(\sup _{(t,x) \in [0,T]\times \mathbb{T}^{d}} \left\| \mathcal{G}^{t,x} \right\|_{C^{2,1}(\mathcal{B}_{\mathbf{r}})}\) in terms of the Fréchet derivatives of \(W \mapsto \rho_W\). Using \eqref{eq:discrete-to-continuous}, we have the following identities:
\begin{align}
    &\sup_{t,x} \left| \mathcal{G}^{t,x}(W) \right| =\left\| \rho_W \right\|_{L^{\infty}([0,T];L^{\infty})} \label{eq:local-reg-term1}, \\
    &\sup_{t,x} \left\| \nabla \mathcal{G}^{t,x}(W) \right\|_D = \sup_{h\in \mathbb{R}^{d} : \| h \|_D=1} \left\| D \rho_W [h] \right\|_{L^{\infty}([0,T];L^{\infty})} \label{eq:local-reg-term2},\\
    &\sup_{t,x} \left\| \nabla ^{2}\mathcal{G}^{t,x}(W) \right\|_{D} = \sup_{h_1,h_2 \in \mathbb{R}^{d} : \| h_1 \|_D=\| h_2 \|_D=1} \left\| D^{2}\rho_W [h_1,h_2] \right\|_{L^{\infty}([0,T];L^{\infty})}. \label{eq:local-reg-term3}
\end{align} 
Recall that by the Sobolev embeddings, \(L^{2}([0,T];H^{\gamma+1}) \cap H^{1}([0,T];H^{\gamma-1}) \hookrightarrow L^{\infty}([0,T];L^{\infty})\) for any \(\gamma > \frac{d}{2}\). For simplicity, we pick \(\gamma=d\). Therefore, modulo a constant depending only on \(d\) and \(T\), the norms in the display above can be replaced by \(L^{2}([0,T];H^{d+1}) \cap H^{1}([0,T];H^{d-1})\). Then, \eqref{eq:local-reg-term1} is bounded due to \cite[Theorem 5]{mckv}, and \eqref{eq:local-reg-term2} and \eqref{eq:local-reg-term3} are bounded due to Lemma \ref{lem:bound-on-derivatives-of-forward-map}. Recall that \(\left\| W \right\|_{W^{2,\infty}} \leq \bar{M}(d,M)\) by \eqref{eq:Br-contained-in-W2infty-ball}, therefore the final constant depends on \(d,T,M, \left\| \phi \right\|_{H^{d}}\).

For the final Lipschitz term, let \(X = L^{2}([0,T];H^{d+1}) \cap H^{1}([0,T];H^{d-1})\) and write, as before,
\[
\sup_{t,x} \left\| \nabla ^{2}\mathcal{G}^{t,x}(W_2) - \nabla ^{2}\mathcal{G}^{t,x}(W_1) \right\|_{D} \lesssim_{d,T} \sup_{h_1,h_2 \in \mathbb{R}^{d} : \| h_1 \|_D=\| h_2 \|_D=1} \left\| (D^{2}\rho_{W_2} - D^{2}\rho_{W_1}) [h_1,h_2] \right\|_{X}.
\] Now, we use the mean value theorem. By the fundamental theorem of calculus in Banach spaces (e.g. \cite[Proposition A.2.3]{LiuRockner_introSPDEs}), 
\[
D^{2}\rho_{W_2} - D^{2}\rho_{W_1} = \int_{0}^{1} D^{3}\rho_{W_1 + (W_2 - W_1)u} [W_2 - W_1]  \mathop{}\!\mathrm{d}u
\] with convergence in \(X\), by Theorem \ref{thm:linearisation-mckean-vlasov}. But since \(W_1, W_2 \in \mathcal{B}_{\mathbf{r}}\), equation \eqref{eq:Br-contained-in-W2infty-ball} shows that
\[
\left\| W_1 + (W_2 - W_1)u \right\|_{W^{2,\infty}} \lesssim \bar{M}.
\] Taking norms in \(X\), using the triangle inequality and Lemma \ref{lem:bound-on-derivatives-of-forward-map} with \(k=3\) and \(\gamma=d\), we obtain
\[
\left\| (D^{2}\rho_{W_2} - D^{2}\rho_{W_1}) [h_1,h_2] \right\|_{X} \leq C(d,T,M,\| \phi \|_{H^{d}}) \left\| h_1 \right\|_{D} \left\| h_2 \right\|_{D} \left\| W_2 - W_1 \right\|_{D}
\] which concludes the proof upon taking suprema over \(h_1,h_2, W_1,W_2\).

\end{proof}

\subsection{Main results}

\subsubsection{Contraction rates}

We first introduce a rate \(\delta_N\), together with its associated exponent \(\eta < 1\):  
\begin{equation}\label{eq:def-slow-rate}
    \delta_N \coloneqq N^{-\frac{\alpha+1}{2(\alpha+1)+d}}, \quad \eta \coloneqq \frac{\beta-2}{\beta} - \frac{3\zeta}{2(\alpha+1)}.
\end{equation} 
Note that due to the choice of constants in \eqref{eq:constants}, \(\eta\) is (strictly) positive. Note also that the rate is such that \(N \delta_N ^{2}\) tends to \(+\infty\).   

We consider a truncated Gaussian prior similar to \cite[Example 2]{mckv} supported on \(\dot{E}_K\):
\begin{equation}\label{eq:prior}
  W(x) = \frac{1}{\sqrt{N}\delta_N}\sum_{\substack{\left| k \right| \leq K \\ k \neq 0}} \frac{1}{ ( 1 + \left| k \right|^{2} )^{(\alpha+1)/2}} g_k \tau_k(x), \qquad g_k \overset{\mathrm{iid}}{\sim} \mathcal{N}(0,1),
\end{equation} 
with \(\tau_k\) from \eqref{eq:trig-basis} and let \(\Pi = \Pi_N \coloneqq \mathrm{Law}(W)\). We observe data from a random design regression model
\begin{equation}\label{eq:data}
    Y_i=\rho_W\left(t_i, X_i\right)+\varepsilon_i, \quad \varepsilon_i \overset{\mathrm{iid}}{\sim} \mathcal{N}(0,1), \quad i=1, \ldots, N
\end{equation}  
where \(\rho_W\) denotes the unique solution to the PDE \eqref{eq:mckv} with initial condition $\phi$, where \(X_i \overset{\mathrm{iid}}{\sim} \mathrm{Unif}(\mathbb{T}^{d})\) and \(t_i \overset{\mathrm{iid}}{\sim} \mathrm{Unif}([0,T])\). The law of the data vector \(Z_N \coloneqq \{(Y_i,t_i,X_i)\}_{i=1}^{N}\) under the parameter \(W\) will be denoted by \(P_W^{N}\). From this data \(Z_N\), the posterior distribution then arises as as in Section 1.2.3 in \cite{nickl2023bayesian},
\begin{equation}\label{eq:posterior}
    \mathop{}\!\mathrm{d} \Pi\left(W \mid Z_N\right) = \frac{e^{\ell_N(W)} \mathop{}\!\mathrm{d} \Pi(W)}{\int_{\dot{E}_K} e^{\ell_N(w)} \mathop{}\!\mathrm{d} \Pi(w)},  \quad \ell_N(W)\coloneqq-\frac{1}{2} \sum_{i=1}^N\left|Y_i-\rho_W\left(t_i, X_i\right)\right|^2,~~~W \in \dot{E}_K.
\end{equation}    

We now introduce two assumptions that one should view as conditions on the couple \((K,N)\) as well as the truth \(W_0\). In high-dimensional settings (\(K\) and therefore \(D(K)\) large), it is natural to expect that \(W_0\) is well approximated by its projection \(W_{0,K}\) onto \(\dot{E}_K\). This can indeed be checked if, for example, we choose the cut-off \(K\) to grow sufficiently fast with \(N\), see Remark \ref{remark:choice-K}, or if \(W_0\) is itself band-limited. Since can be checked by other means we choose to separate the control of the bias from the rest of the argument and present it as an assumption, but see Remark \ref{remark:choice-K}.

\begin{assumption}[Control of the bias]\label{assumption:control-bias}
    Given some \(c_{\mathrm{err}} > 0\), assume 
    \begin{equation}\label{eq:control-bias-forward}
        \left\| \rho_{W_0}-\rho_{W_{0,K}} \right\|_{L^{2}(\mathcal{X},\lambda)} \leq \frac{\delta_N}{2},
    \end{equation} 
    as well as 
    \begin{equation}\label{eq:control-bias-inverse}
        \left\| W_0-W_{0,K} \right\|_{L^{2}} \leq c_{\mathrm{err}} \left( \delta_N \right)^{\eta}.    
    \end{equation}  
\end{assumption}

\begin{assumption}[Prior smoothness]\label{assumption:upper-bound-K}
    Given some \(c_{\mathrm{pr}}>0\), assume  
    \[
    D \leq c_{\mathrm{pr}} N \delta_N^{2}.
    \] 
\end{assumption}

First we prove a suitable version of \cite[Theorem 3]{mckv} to fit the current setting. Namely, we show that the posterior contracts around the ground truth even when restricted to \(W\) of uniformly (in \(K,N\)) bounded Sobolev norm, the so-called `regularisation sets'.

\begin{theorem}[Posterior contraction]\label{thm:posterior-contraction}
    Choose \(\alpha, \beta\) and \(\zeta\) as in \eqref{eq:constants}. Let the initial condition \(\phi\) satisfy Assumptions \ref{assumption:regulatity-mckv} and \ref{assumption:phi} for this choice of \(\beta\) and \(\zeta\). Consider the space \(\dot{E}_K\) of mean-zero trigonometric polynomials of degree at most \(K\) from \eqref{eq:def-E_K}, and the rate \((\delta_N,\eta)\) from \eqref{eq:def-slow-rate}. Let the posterior distribution \(\Pi(\ \cdot \mid Z_N)\) arise as in \eqref{eq:posterior} from prior \(\Pi_N\) in \eqref{eq:prior} and data \(Z_N = \{(Y_i,t_i,X_i)\}_{i=1}^{N}\) in \eqref{eq:data}. Let \(a,M >0\), \(b\) larger than some universal constant, and \(W_0 \in \dot{H}^{\alpha+1}, \left\| W_0 \right\|_{H^{\alpha+1}} \leq M, K \geq 1,N \in \mathbb{N},\) be such that Assumptions \ref{assumption:control-bias} and \ref{assumption:upper-bound-K} hold.
    
    Then, there exists a constant \(m>0\) depending on \(a,b,c_{\mathrm{err}},c_{\mathrm{pr}},\alpha,\beta,d,c_{\ast}, T,M, \phi_{\mathrm{min}},\left\| \phi \right\|_{H^{\beta}}\) such that 
    \[
    P_{W_0}^{N} \left( \Pi(W \in \dot{E}_K: \left\| W-W_0 \right\|_{D} \leq m (\delta_N)^{\eta}, \left\| W \right\|_{H^{\alpha+1}} \leq m | Z_N) \geq 1 - e^{-aN \delta_N^{2}} \right) \geq 1 - e^{-bN \delta_N^{2}}. 
    \]  
\end{theorem} 

\begin{remark}[Choice of a cutoff sequence \(K=K_N\)]\label{remark:choice-K} \normalfont

    Let \(M>0\) and \(W_0 \in \dot{H}^{\alpha+1}\) such that \(\| W_0 \|_{H^{\alpha+1}}\leq M\), and recall \(L(d,T,M,\| \phi \|_{H^{d}})\) the Lipschitz constant from Theorem \ref{thm:lipschitz-continuity-forward-map}. Define 
    \[
    c_{\min} \coloneqq (2ML)^{\frac{1}{\alpha+1+\beta+1}}, \quad K_N = c (N \delta_N^{2})^{1/d} \text{ for any } c \geq c_{\min}
    \]
    Then, one checks using the choice of constants in \eqref{eq:constants} as well as Theorem \ref{thm:lipschitz-continuity-forward-map} that for any \(N \in \mathbb N\), the couple \((N,K_N)\) satisfies Assumptions \ref{assumption:control-bias} and \ref{assumption:upper-bound-K} with choices of \(c_{\mathrm{pr}}\) and \(c_{\mathrm{err}}\) depending only on \(\alpha\), \(d\) and \(c\). Therefore, the conclusions of Theorem \ref{thm:posterior-contraction} hold with \(K=K_N\), \(D=D_N\), and for a constant \(m>0\) depending only on \(a,b,\alpha,\beta,d,c_{\ast},T,M,\phi_{\mathrm{min}},\| \phi \|_{H^{\beta}}\).
\end{remark}

\begin{proof}[Proof of Theorem \ref{thm:posterior-contraction}] Let \(a,b,M\) be as in the theorem. We first prove a forward contraction rate, as in \cite[Theorem 1]{mckv}. Namely, there exist \(\tilde{m}(a,b,c_{\mathrm{pr}},\alpha,\beta,d, T,M,\left\| \phi \right\|_{H^{\beta}})>0\) large enough such that
    \[
        P_{W_0}^{N} \left( \Pi( \mathcal{E}_N | Z_N) \geq 1 - e^{-aN \delta_N^{2}} \right) \geq 1 - e^{-bN \delta_N^{2}} 
    \]
    where 
    \[
    \mathcal{E}_N \coloneqq \left\{ W \in \dot{E}_K: \left\| \rho_W-\rho_{W_0} \right\|_{L^{2}([0,T];L^{2})} \leq \tilde{m} \delta_N, \left\| W \right\|_{H^{\alpha+1}} \leq \tilde{m} \right\}. 
    \]
    This bound is a consequence of \cite[Theorem 2.2.2]{nickl2023bayesian} and more specifically \cite[Exercise 2.4.3]{nickl2023bayesian}. The boundedness of \(\rho_W\) is proved as in \cite[Theorem 1]{mckv}, and the Lipschitz continuity of \(W \mapsto \rho_W : \dot{L}^{2}\to L^{2}(\mathcal X)\) follows from Theorem \ref{thm:lipschitz-continuity-forward-map}. Moreover, the regularisation sets can be chosen as above since if \(\| W_1 \|_{L^{2}} \leq M \delta_N\), then for \(\tilde{m}\) large enough,
    \[
    \| W_1 \|_{H^{\alpha+1}} \leq K^{\alpha+1} \| W_1 \|_{L^{2}} \lesssim_d c_{\mathrm{pr}} (N \delta_N^{2})^{\frac{\alpha+1}{d}} M \delta_N  = c_{\mathrm{pr}} M \leq \tilde{m}
    \] where we used that \(W_1 \in \dot{E}_K\) as well as Assumption \ref{assumption:upper-bound-K}, equations (\ref{terminator}), (\ref{eq:equivalence-D-K^d}), and the definition of \(\delta_N\) in \eqref{eq:def-slow-rate}.
   That these events have probability at least \(1 - e^{-bN \delta_N^{2}}\) follows from inspection of the proofs of \cite[Theorems 1.3.2 and 2.2.2]{nickl2023bayesian} as well as \cite[Exercise 5.4.2]{nickl2023bayesian}: in that exercise, \(b\) can be taken as large as desired by letting \(x=b N \sigma^{2} + \log(2)\) in Bernstein's inequality, provided \(K=K(b)\) is large enough. The type-one testing errors in the proof of \cite[Theorem 1.3.2]{nickl2023bayesian} are controlled at the required exponential rate using \cite[Theorem 7.1.4]{nickl2016mathematical} for \(m\) large enough. 
 
    Then, we apply the stability estimate in Theorem \ref{thm:mckv-stability-estimate} to transfer the rate to \(\| W-W_0 \|_{L^{2}}\). We use interpolation in time as in the proof of \cite[Theorem 3]{mckv} to obtain that on the event \(\mathcal{E}_N\) from above,
    \[
    \left\| W-W_0 \right\|_{L^{2}} \lesssim_{d,\beta,T} \left\| W_0-W_{0,K} \right\|_{L^{2}} + \tilde{m}S K^{3\zeta/2} \delta_N^{\frac{(\beta-2)}{\beta}}.
    \] 
    Now use Assumption \ref{assumption:control-bias}, equation \eqref{eq:control-bias-inverse}, and Assumption \ref{assumption:upper-bound-K}, to conclude that for \(m\) large enough, 
    \[
    \left\| W-W_0 \right\|_{L^{2}} \leq m (\delta_N)^{\eta} 
    \] since by definition of \(\eta\) in \eqref{eq:def-slow-rate}, \((N \delta_N^{2})^{\frac{3\zeta}{2d}} \delta_N^{\frac{\beta-2}{\beta}} = \delta_N^{\eta}\). 
\end{proof}

\begin{remark}[Faster contraction rates] \normalfont
    Note that \cite[Theorem 3]{mckv} obtain a faster contraction rate than in the previous proof. While we will not take advantage of this below, let us remark here that our methods can obtain this rate as well, in fact an even slightly faster one via the Lipschitz estimate in Theorem \ref{thm:lipschitz-continuity-forward-map}. Specifically, if one defines a new rate and a new exponent by
    \[
        \tilde{\delta}_N \coloneqq N^{-\frac{\alpha+1+\beta+1}{2(\alpha+1)+2(\beta+1)+d}}, \quad \tilde{\eta} \coloneqq \frac{\beta-2}{\beta} - \frac{3\zeta}{2(\alpha+1)+2(\beta+1)}
    \] 
    then \(W \mapsto \rho_W\) is still Lipschitz continuous from \(\dot{H}^{-(\beta +1)}\) to \(L^{2}(\mathcal X)\) by Theorem \ref{thm:lipschitz-continuity-forward-map}, and if \break \(\| W_1 \|_{H^{-(\beta+1)}} \leq M \tilde{\delta}_N\), then it still holds that for \(\tilde{m}\) large enough,
    \[
    \| W_1 \|_{H^{\alpha+1}} \leq K^{\alpha+1+\beta+1} \| W_1 \|_{H^{-(\beta+1)}} \lesssim_d c_{\mathrm{pr}} (N \tilde{\delta}_N^{2})^{\frac{\alpha+1+\beta+1}{d}} M \tilde{\delta}_N  = c_{\mathrm{pr}} M \leq \tilde{m}
    \] where we used that \(W_1 \in \dot{E}_K\) as well as Assumption \ref{assumption:upper-bound-K} (with \(\tilde{\delta}_N\) in place of \(\delta_N\)) and equations \eqref{terminator}, \eqref{eq:equivalence-D-K^d}. Therefore, provided Assumptions \ref{assumption:control-bias} and \ref{assumption:upper-bound-K} hold with \((\tilde{\delta}_N,\tilde{\eta})\) in place of \((\delta_N,\eta)\), and that the prior \eqref{eq:prior} is rescaled by \(\tilde{\delta}_N\) instead of \(\delta_N\), the conclusions of Theorem \ref{thm:posterior-contraction} hold with the faster rate \((\tilde{\delta}_N)^{\tilde{\eta}}\). In fact this argument also implies the posterior contraction rate $\tilde \delta_N$ for $\|\rho_W-\rho_{W_0}\|_{L^2_TL^2}$ which we believe to be minimax optimal.
\end{remark}

\subsubsection{Log-concave approximation of the posterior}\label{subsubsec:log-concave-approximation}

We recall here the definition of the surrogate posterior \(\tilde{\Pi}(\cdot \mid Z_N)\) exactly as in \cite[Section 5.1.1]{nickl2023bayesian}. Let 
\begin{equation}\label{eq:def-mathbf-r}
    \mathbf{r} \coloneqq r D^{-w} \text{ for some } 0< r \leq r_{\mathrm{max}}
\end{equation} 
with \(r_{\mathrm{max}}\) as in Theorem \ref{thm:local-average-curvature} and \(w\) as in \eqref{eq:constants}. Then, let \(W_{\text{init}} \in \dot{E}_K\) be a `warm-start' initialiser such that 
\begin{equation}\label{eq:Winit}
\left\| W_{\text{init}} - W_{0,K}  \right\|_D \leq \frac{\mathbf{r}}{8}.
\end{equation} 
We require two auxiliary functions: \(g_{\mathbf{r}}\) (with strongly convex tails), and \(\alpha_{\mathbf{r}}\) (a cutoff function) defined as follows. For some smooth and symmetric (about \(0\)) \(\varphi: \mathbb{R} \rightarrow[0, \infty)\) supported in \([-1,1]\) and integrating to \(\int_{\mathbb{R}} \varphi(x) \mathop{}\!\mathrm{d} x=1\), let us first define mollifiers \(\varphi_h(x):= h^{-1} \varphi(x / h), h>0\). Then, define \(\tilde{\gamma}_{\mathbf{r}}, \gamma_{\mathbf{r}}: \mathbb{R} \rightarrow \mathbb{R}\) by

\[
\tilde{\gamma}_{\mathbf{r}}(t):=\left\{\begin{array}{ll}
0 & \text { if } t<5 \mathbf{r} / 8  \\
(t-5 \mathbf{r} / 8)^2 & \text { if } t \geq 5 \mathbf{r} / 8
\end{array}, \quad \gamma_{\mathbf{r}}(t):=\left[\varphi_{\mathbf{r} / 8} \ast \tilde{\gamma}_{\mathbf{r}}\right](t)\right. .
\] 
Further, let \(\alpha:[0, \infty) \rightarrow[0,1]\) be smooth and satisfy \(\alpha(t)=1\) for \(t \in[0,3 / 4]\) and \(\alpha(t)=0\) for \(t \in[7 / 8, \infty)\). Finally we define \(g_{\mathbf{r}}: \dot{E}_K \rightarrow [0, \infty)\) and \(\alpha_{\mathbf{r}}: \dot{E}_K \rightarrow[0,1]\) as

\[
g_{\mathbf{r}}(W):=\gamma_{\mathbf{r}}\left(\left\|W-W_{\text {init }}\right\|_{D}\right), \quad \alpha_{\mathbf{r}}(W)=\alpha\left(\frac{\left\|W-W_{\text {init }}\right\|_{D}}{\mathbf{r}}\right) .
\]
Let \(\lambda_{\min} \coloneqq \max\big(N \log(N) / \mathbf{r}^{2}, CN(c_1 + 1)(1+\mathbf{r}^{-2})\big)\) where \(C\) is chosen as in \cite[Proposition 5.1.2]{nickl2023bayesian} and \(c_1\) was introduced in Proposition \ref{prop:local-regularity}. Now for \(\ell_N\) as in \eqref{eq:posterior} and some \(\lambda \geq \lambda_{\min}\), define the surrogate likelihood function

\[
\tilde{\ell}_N(W)=\alpha_{\mathbf{r}}(W) \ell_N(W)- \lambda g_{\mathbf{r}}(W), \quad W \in \dot{E}_K,
\] which `convexifies' \(-\ell_N\) in the `tails'. The surrogate posterior \(\tilde{\Pi}(\cdot \mid Z_N)\) is then defined via its density
\begin{equation}\label{eq:surrogate-posterior}
    \mathop{}\!\mathrm{d}\tilde{\Pi}(W \mid Z_N) = \frac{e^{\tilde{\ell }_N(W)} \mathop{}\!\mathrm{d}\Pi(W)}{\int_{\dot{E}_K} e^{\tilde{\ell }_N(w)} \mathop{}\!\mathrm{d}\Pi(w)},
\end{equation}
which is globally strongly log-concave. Note that this surrogate depends on the choice of the initialiser \(W_{\mathrm{init}}\), which can be taken to be \(W_{\mathrm{init}} = W_{0,K}\) in the following theorem. For applications to MCMC runtimes below, feasible initialisers need to be found. This question is of independent interest and will be studied elsewhere.

\begin{theorem}[Log-concave Wasserstein approximation]\label{thm:log-concave-approximation}
    Choose \(\alpha, \beta, \zeta\) and \(w\) as in \eqref{eq:constants}. Let the initial condition \(\phi\) satisfy Assumptions \ref{assumption:regulatity-mckv} and \ref{assumption:phi} for this choice of \(\beta\) and \(\zeta\). Let the posterior distribution \(\Pi(\ \cdot \mid Z_N)\) arise as in \eqref{eq:posterior} from prior \(\Pi_N\) in \eqref{eq:prior} and data \(Z_N = \{(Y_i,t_i,X_i)\}_{i=1}^{N}\) in \eqref{eq:data}. Recall the rate \((\delta_N,\eta)\) from \eqref{eq:def-slow-rate} and let the surrogate posterior \(\tilde{\Pi}\left(\cdot \mid Z_N\right)\) be given as in \eqref{eq:surrogate-posterior} with some \(W_{\mathrm{init}}\) verifying \eqref{eq:Winit}. Let \(M>0, \left\| W_0 \right\|_{H^{\alpha+1}} \leq M, K \geq 1, N \in \mathbb{N}, W_0 \in \dot{H}^{\alpha + 1}\) be such that be such that Assumptions \ref{assumption:control-bias} and \ref{assumption:upper-bound-K} hold, and \mbox{\(\| \rho_{W_0} - \rho_{W_{0,K}}\|_{L^{2}([0,T];L^{2})} \leq \mathbf{r}\)} defined in \eqref{eq:def-mathbf-r}. 

    Then, there exist \(C_{\mathrm{pr}}(d,w,\alpha,\beta,\zeta,r)\), some \(b,c >0\) that are uniform in \(W_0, K, N\), and an universal constant \(C>0\) such that if \(c_{\mathrm{pr}}\) in Assumption \eqref{assumption:upper-bound-K} verfies \(c_{\mathrm{pr}} \leq C_{\mathrm{pr}}\),
    \[
    P^{N}_{W_0} \left( \mathcal{W}_2^2\left(\Pi\left(\ \!\cdot \mid Z_N\right), \tilde{\Pi}\left(\ \!\cdot \mid Z_N\right)\right) \leq e^{-N \delta_N^{2}} \right) \geq 1 - C \left( e^{-bN \delta_N^{2}}+e^{-cND^{-12\zeta/d}} \right) \xrightarrow[N\to \infty]{} 1
    \]
    with \(\mathcal{W}_{2}\) the Wasserstein-2 distance defined in \eqref{eq:def-wasserstein}.
\end{theorem}

\begin{proof}[Proof of Theorem \ref{thm:log-concave-approximation}]
    We will check that the assumptions of \cite[Theorem 5.1.3]{nickl2023bayesian} are verified for \(\mathbf{r} = rD^{-w}\) where \(0< r \leq r_{\mathrm{max}}\) as in Theorem \ref{thm:local-average-curvature}. We start by verifying the subconditions of Condition 5.1.1 in \cite{nickl2023bayesian} with \(\Theta = \mathscr{R} = H^{\alpha+1}\) with \(\alpha+1 > 2 + d/2 > 1 + d/2\).

    Condition 2.1.1 is satisfied for \(\kappa=\beta+1\) by \cite[Theorem 5]{mckv} (combined with the usual Sobolev embeddings) and Theorem \ref{thm:lipschitz-continuity-forward-map}, and therefore holds with \(\kappa=0\) as well.

    Condition 2.1.4 does not hold formally due to the deconvolution nature of the problem. However, it can be replaced by Theorem \ref{thm:posterior-contraction} which provides a contraction rate of \((\delta_N)^{\eta}\), with also \(0 < \eta \leq 1\) due to the choice of constants in \eqref{eq:constants}. In the proof of \cite[Theorem 5.1.3]{nickl2023bayesian}, it suffices to replace the first display in the bound of Term III by the result of Theorem \ref{thm:posterior-contraction} to obtain to the same conclusion.

    Condition 3.2.1 holds for \(\kappa_2 = 0\) as seen in Proposition \ref{prop:local-regularity}. 
    
    Condition 3.2.2 holds for \(\kappa_0 = 6\zeta/d\) and \(\kappa_1 = 0\) by Theorem \ref{thm:local-average-curvature} and the remark below \cite[Condition 3.2.2]{nickl2023bayesian}, also since \mbox{\(\| \rho_{W_0} - \rho_{W_{0,K}}\|_{L^{2}([0,T];L^{2})} \leq \mathbf{r}\)} by assumption.
    
    We now check equation (5.5) in \cite[Condition 5.1.1]{nickl2023bayesian}. By the choice of constants in \eqref{eq:constants} and Assumption \ref{assumption:upper-bound-K},
    \[
    \log(N)^{2} (\delta_N)^{\eta} D^{w} \xrightarrow[N \to +\infty]{} 0, \quad \log(N)^{2} \delta_N D^{w+6\zeta/d} \xrightarrow[N \to +\infty]{} 0
    \] and in particular there exists \(N_{\min}(d,w,\alpha,\beta,\zeta,r)\) such that for all \(N \geq N_{\min}\), \(\mathbf{r} \geq \tilde{\delta}_N\log N\). Then, one can choose \(C_{\mathrm{pr}}(d,w,\alpha,\beta,\zeta,r)\) small enough to control the terms \(N \leq N_{\mathrm{min}}\) so that whenever \(c_{\mathrm{pr}} \leq C_{\mathrm{pr}}\), \(\mathbf{r} \geq \tilde{\delta}_N\log N\) hold for all \(N \in \mathbb{N}\). 

    Finally, the condition \(\lambda \geq N \log(N) D^{\kappa_1}/\mathbf{r}^{2}\) in \cite[Theorem 5.1.3]{nickl2023bayesian} (in this reference, \(\lambda\) is denoted by \(K\)) is satisfied by construction of the surrogate in \eqref{eq:surrogate-posterior}. Note also that \(ND^{-12\zeta/d} \xrightarrow[N \to \infty]{} +\infty\) by Assumption \ref{assumption:upper-bound-K} and the choice of constants in \eqref{eq:constants}.  
\end{proof}

\subsubsection{Polynomial mixing time bounds for warm start ULA}

We start with a brief description of the Unadjusted Langevin Algorithm (ULA),  studied in Theorem \ref{thm:wasserstein-mixing-time} below. Given a target measure with positive density \(p\) on \(\mathbb{R}^{D}\), the algorithm comes from a Euler-Maruyama discretisation of the Langevin diffusion 
\[
\mathop{}\!\mathrm{d} L_t = \nabla \log p(L_t) \mathop{}\!\mathrm{d}t + \sqrt{2} \mathop{}\!\mathrm{d} B_t
\] 
which is a continuous-time Markov process with invariant measure having density \(p\). In our case, the density of interest is \(p = \mathop{}\!\mathrm{d} \tilde{\Pi}_N\), i.e. the density of the surrogate posterior distribution in \eqref{eq:surrogate-posterior}. Since this density is log-concave, we can obtain mixing time bounds using the existing theory for ULA as in \cite{DurmusMoulines2019}. Then, we can use that the surrogate posterior approximates the true posterior in Wasserstein distance (Theorem \ref{thm:log-concave-approximation}) to obtain mixing time bounds for the true posterior as well. This is the object of Theorem \ref{thm:wasserstein-mixing-time} below.

Under the usual identification of \(\dot{E}_K\) with \(\mathbb{R}^{D}\), we can see the prior \(W\) in \eqref{eq:prior} as a \(D\)-dimensional Gaussian vector with covariance matrix 
\begin{equation}\label{eq:prior-gaussian-vector}
\Sigma \coloneqq (N \delta_N^{2})^{-1}\Sigma', \quad W \sim \mathcal{N}(0,\Sigma)    
\end{equation} 
where \(\Sigma'\) is a diagonal matrix with diagonal entries \((1+| k |^{2})^{-(\alpha+1)}\) for \(k \in \mathbb{Z}^{d} \setminus\{0\}\), \(\left| k \right| \leq K\), counted with multiplicities. The surrogate posterior probability density is then given by
\[
\mathop{}\!\mathrm{d} \tilde{\Pi}_N (W) = \frac{1}{Z} e^{\tilde{\ell}_N(W)-\frac{1}{2} W^{T} \Sigma^{-1}W}, \quad W \in \mathbb{R}^{D}
\]     
for normalizing constant \(Z\). For some initializer \(W_{\mathrm{init}}\) verifying \eqref{eq:Winit} and step size \(\gamma > 0\), we define the ULA Markov chain \((\vartheta_k)_{k \geq 0}=(\vartheta_k(W_{\mathrm{init}}, \gamma))_{k \geq 0}\) by 
\begin{equation}\label{eq:ula-markov-chain}
   \begin{aligned}
\vartheta_0 & =W_{\mathrm{init}}, \\
\vartheta_{k+1} & =\vartheta_k+\gamma\left[\nabla \tilde{\ell}_N\left(\vartheta_k\right)-\Sigma^{-1} \vartheta_k\right]+\sqrt{2 \gamma} \xi_{k+1} \quad \text { for } k=0,1, \ldots
\end{aligned} 
\end{equation} 
where \(\xi_k \stackrel{\text { i.i.d. }}{\sim} N\left(0, I_{D}\right)\). Note that the drift term in brackets is exactly \(\nabla \log (\mathop{}\!\mathrm{d} \tilde{\Pi}_N(W))\). 

\begin{theorem}[Wasserstein mixing time]\label{thm:wasserstein-mixing-time}
    Suppose that the hypotheses of Theorem \ref{thm:log-concave-approximation} hold, and that there exists a feasible initializer \(W_{\mathrm{init}}\) satisfying \eqref{eq:Winit} for the Markov chain \((\vartheta_k)_{k \geq 0}\) in \eqref{eq:ula-markov-chain}. Assume also that Assumption \ref{assumption:upper-bound-K} holds for some \(0<c_{\mathrm{pr}} \leq C_{\mathrm{pr}}\) from Theorem \ref{thm:log-concave-approximation}.
    
    Then, for any desired accuracy \(\varepsilon >0\), there exists a step size \(\gamma=\gamma(\varepsilon,D,N)\) and an integer \mbox{\(k_{\mathrm{mix}}=O(\varepsilon^{-b_1} N^{b_2} D^{b_3})\)} scaling polynomially in \(\varepsilon^{-1},D\) and \(N\), such that for all \(k \geq k_{\mathrm{mix}}\),
    \[
    P_{W_0}^{N} \left( \mathcal{W}_2^{2}(\mathcal{L}(\vartheta_k), \Pi(\ \! \cdot \mid Z_N)) \leq e^{-N^{d/(2(\alpha+1)+d)}} + \varepsilon  \right) \geq 1 - C \left( e^{-bN \delta_N^{2}} + e^{-cND^{-12\zeta/d}} \right) \xrightarrow[N\to \infty]{}1 
    \] for some \(b_1, b_2, b_3, b, c, C>0\).  
\end{theorem} 

\begin{proof}[Proof of Theorem \ref{thm:wasserstein-mixing-time}]
    Note that since Assumption \ref{assumption:upper-bound-K} with the choice of constants in \eqref{eq:constants} implies that \(D \lesssim ND^{-12\zeta/d}\), the conclusions of \cite[Theorem 5.2.1]{nickl2023bayesian} hold (also since our choice of \(\lambda\) in the construction of the surrogate is large enough). We choose the event \(\mathscr{E} = \mathscr{E}_{\mathrm{conv}} \cap \mathscr{E}_{\mathrm{wass}}(\rho)\) with \(\rho=e^{-N \delta_N^{2}}\) and apply this theorem. By taking \(\gamma\) small enough and then \(k\) large enough, we can make the right-hand side of equation (5.21) in \cite[Theorem 5.2.1]{nickl2023bayesian} smaller than \(\rho+\varepsilon\). The lower bound \(k_{\mathrm{mix}}\) grows polynomially in \(N\) and \(D\) since the eigenvalues of \(\Sigma\) defined in \eqref{eq:prior-gaussian-vector} scale polynomially in \(N\) and \(D\), see also \eqref{eq:equivalence-D-K^d}.

    Therefore, the proof will be complete if we can show that \(\mathscr{E}\) holds with the required probability. By the same observation that \(D \lesssim ND^{-12\zeta/d}\), the conclusions of \cite[Theorem 3.2.3]{nickl2023bayesian} hold as well. Therefore, using also Theorem \ref{thm:log-concave-approximation},
    \[
    P_{W_0}^{N}(\mathscr{E}_{\mathrm{conv}} \cap \mathscr{E}_{\mathrm{wass}}(e^{-N \delta_N^{2}})) \geq 1 - C \left( e^{-bN \delta_N^{2}} + e^{-cND^{-12\zeta/d}} + e^{-N/8} \right).
    \] 
    The last term is negligible and can be absorbed in the constant \(C\). This proves the theorem.
\end{proof}

\begin{remark} \normalfont
    Theorem \ref{thm:wasserstein-mixing-time} also informs the complexity of computing the posterior mean vector \break \(\Pi(W \mid Z_N)\). In fact, provided a feasible initialiser exists, any Lipschitz functional of the posterior can be approximated in polynomial time by ergodic averages along the chain \(\vartheta_k\) of the form 
    \[
    \frac{1}{J} \sum_{k=J_{\mathrm{in}}+1}^{J_{\mathrm{in}}+J} H(\vartheta_k)
    \] where \(J_{\mathrm{in}}\) is a burn-in period and \(H\) is a Lipschitz functional of interest. The details are laid out in \cite[Theorem 5.2.2]{nickl2023bayesian}.
\end{remark}

\appendix

\section{PDE estimates for a linearised McKean--Vlasov equation}\label{appendix:parabolic-regularity}

Throughout this section, we consider existence and regularity for the following Cauchy problem on \mbox{\(\mathcal{X} = [0,T] \times \mathbb{T}^{d}\)} for some fixed \(T>0\):
\begin{equation}\label{eq:linPDE}
    \left( \partial _{t} - \mathcal{L}_{W} \right) u = f, \quad u(0) = 0
\end{equation} 
with \textit{linear} differential operator $$\mathcal{L}_W= \Delta [\cdot ] + \nabla \cdot ([\cdot ] \nabla W \ast \rho_W) + \nabla \cdot (\rho_W \nabla W \ast [\cdot ]).$$ Here we assume \(\| W \|_{W^{2,\infty}} \leq M\) for some fixed \(M>0\), and $\rho_W$ is the given solution to the non-linear PDE (\ref{eq:mckv}) -- in fact $\rho_W$ could be replaced by any function with an appropriate bound on its Sobolev norm. The source function \(f\) depends on time and space, with conditions to be specified. We start with an auxiliary lemma.

\begin{lemma}\label{lem:trilin-op-T-bound}
   Recall from \eqref{eq:def-Tcal} that the trilinear operator \(\mathcal{T}\) is defined by \(\mathcal{T}(f,V,g) \coloneqq \nabla \cdot \left( f \nabla V \ast g \right) \). For any \(f,V,g\) such that the expressions below make sense, we have the following estimates:

   \begin{enumerate}
        \item For any integers \(k \geq 0\), \(0 \leq l \leq k\), and any \(\nu \in \mathbb R\), 
    \begin{equation}\label{eq:trilin-leibniz}
    \left\| \mathcal{T}(f,V,g) \right\|_{H^{k-1}} \lesssim_{k,d} \mathrm{min} \left\{ \left\| \nabla V \right\|_{W^{l,\infty}} \left\| f \right\|_{H^{k}}\left\| g \right\|_{W^{k-l,1}} , \left\| V \right\|_{H^{\nu}} \left\| f \right\|_{W^{k,\infty}} \left\| g \right\|_{H^{k+1-\nu}} \right\}.
    \end{equation}
    In particular, since \(\mathbb{T}^{d}\) is bounded,
    \[
    \left\| \mathcal{T}(f,V,g) \right\|_{H^{k-1}} \lesssim_{k,d} \left\| \nabla V \right\|_{W^{l,\infty}} \left\| f \right\|_{H^{k}}\left\| g \right\|_{H^{k-l}}.
    \]   

    \item For any real \(\gamma>\frac{d}{2}\) and \(\nu\),
   \begin{equation}\label{eq:trilin-algebra}
    \left\| \mathcal{T}(f,V,g) \right\|_{H^{\gamma-1}} \lesssim_{\gamma,d} \left\| f \right\|_{H^{\gamma }} \left\| g \right\|_{H^{\gamma+1-\nu}} \left\| V \right\|_{H^{\nu}}.
   \end{equation}  

        \item For any real \(\nu\),
\begin{equation}\label{eq:trilin-negative}
   \left\| \mathcal{T}(f,V,g) \right\|_{H^{-2}} \leq \left\| f \right\|_{W^{1,\infty}} \left\| V \right\|_{H^{-\nu}} \left\| g \right\|_{H^{\nu}}.
\end{equation} 
   \end{enumerate}
\end{lemma}

\begin{proof}
   The proof is standard and consists of an application of Fourier-analytical inequalities for convolutions on the torus of Sobolev functions such as 
   \begin{equation}\label{eq:convolution-estimate}
\|u \ast v\|_{H^{a+b}} \le \|u\|_{H^a} \|v\|_{H^b}, ~a,b \in \mathbb R,
   \end{equation}
   and we only sketch the main ideas. First note that for any \(\mu \in \mathbb{R}\),
\[
\left\| \mathcal{T}(f,V,g) \right\|_{H^{\mu -1}} = \left\| \nabla \cdot (f \nabla V \ast g) \right\|_{H^{\mu-1}} \leq \left\| f \nabla V \ast g \right\|_{H^{\mu}}
\] which will always be the first step of the proofs below.
\begin{enumerate}
    \item For the first part of the minimum we can bound
    $$ \left\| f \nabla V \ast g \right\|_{H^{k}} \lesssim_{k,d} \left\| f \right\|_{H^{k}} \left\| \nabla V \ast g \right\|_{W^{k,\infty}} \lesssim_{k,d} \left\| \nabla V \right\|_{W^{l,\infty}} \left\| f \right\|_{H^{k}}\left\| g \right\|_{W^{k-l,1}}  $$ as required. For the second part of the minimum we see
\[
    \left\| f \nabla V \ast g \right\|_{H^{k}} \lesssim_{k,d} \left\| f \right\|_{W^{k,\infty}} \left\| \nabla V \ast g \right\|_{H^{k}} \leq \left\| f \right\|_{W^{k,\infty}} \left\| V\ast g \right\|_{H^{k+1}}
\] 
and conclude using \eqref{eq:convolution-estimate} with \(a=\nu\) and \(b=k+1-\nu\).

    \item Since \(\gamma > d/2\), pointwise multipliers are Lipschitz on $H^\gamma$ and so
\[
\| f \nabla V \ast g \|_{H^{\gamma}} \lesssim_{\gamma,d} \| f \|_{H^{\gamma}} \| V \ast g \|_{H^{\gamma+1}}.
\] Conclude using \eqref{eq:convolution-estimate} with \(a=\gamma+1-\nu\) and \(b=\nu\). 

\item Let \(\psi \coloneqq V \ast g\) and recall that
\[
\| f \nabla \psi  \|_{H^{-1}} = \sup_{\substack{\varphi \in H^{1} \\ \| \varphi \|_{H^{1}} \leq 1}} \left\langle f \nabla \psi, \varphi \right\rangle _{L^{2}}.
\] From integration by parts, the product rule and the Cauchy-Schwarz inequality, we see that  
\[
\left| \left\langle f \nabla \psi, \varphi \right\rangle _{L^{2}} \right| \leq \left| \left\langle \psi \nabla f, \varphi \right\rangle  \right| + \left| \left\langle f \psi, \nabla \varphi \right\rangle  \right| \leq \left\| f \right\|_{W^{1,\infty}} \left\| \psi \right\|_{L^{2}} \left\| \phi \right\|_{H^{1}}.
\] Therefore, after taking the supremum over all \(\varphi \in H^{1}\) we obtain
\[
\left\| f \nabla V \ast g \right\|_{H^{-1}} \leq \left\| f \right\|_{W^{1,\infty}} \left\| V \ast g \right\|_{L^{2}}.
\] Using \eqref{eq:convolution-estimate} with \(a=\nu\) and \(b=-\nu\), the result follows.
\end{enumerate}  
\end{proof}

\begin{theorem}\label{thm:existence-linPDE}
    Let \(M>0, W \in \dot{W}^{2,\infty}\) be such that \(\| W \|_{W^{2,\infty}} \leq M\), and let  \(\rho_W\) appearing in \(\mathcal{L}_W\) in \eqref{eq:linPDE} denote the unique solution to \eqref{eq:mckv}. Let \(\phi\), the initial condition of \(\rho_W\), satisfy Assumption \ref{assumption:regulatity-mckv} for $
    \beta\geq3+d$. Then, for any \(f\in L^{2}([0,T];L^{2})\) there exists a unique strong solution \(u\) of \eqref{eq:linPDE} with
    \[
    u \in L^{2}([0,T];H^{2}) \cap H^{1}([0,T];L^{2}).
    \]
    Moreover, there exists a constant \(C = C \left( d,T,M,\| \phi \|_{H^{1}} \right) \) such that with \(u' = \frac{\mathop{}\!\mathrm{d}u}{\mathop{}\!\mathrm{d}t}\)
    \[
        \left\| u \right\|_{L^{2}([0,T];H^{2})} + \left\| u \right\|_{L^{\infty}([0,T];H^{1})} + \left\|u'\right\|_{L^{2}([0,T];L^{2})} \leq C \left\| f \right\|_{L^{2}([0,T];L^{2})}.
    \] 
\end{theorem}

\begin{proof}
    We explain how to deal with the nonlocal term \(\nabla \cdot (\rho_W \nabla W \ast [\cdot ])\) by deriving the a priori estimates needed for the full proof, mostly adapting \cite[Sec.~7.1.3 Theorem 5]{evans2010partial} to our case. Standard Galerkin approximation arguments then yield existence and regularity in the desired function spaces, see \cite{evans2010partial} for details. We henceforth denote by \(u\) a smooth solution of \eqref{eq:linPDE}. 

    Multiply \eqref{eq:linPDE} by \(u\) and integrate by parts to obtain
    \[
    \frac{1}{2} \frac{\mathrm{d}}{\mathrm{d}t} \left\| u \right\|_{L^{2}}^{2} + \left\| \nabla u \right\|_{L^{2}}^{2} = -\left\langle u \nabla W \ast \rho_W, \nabla u \right\rangle _{L^{2}} -\left\langle \rho_W \nabla W \ast u, \nabla u \right\rangle _{L^{2}} + \left\langle f, u \right\rangle _{L^{2}}.
    \] 
    Using the Cauchy-Schwarz inequality and the proof of Lemma \ref{lem:trilin-op-T-bound} equation \eqref{eq:trilin-leibniz} with \(k=0\) and \(l=0\), we see
\[
    -\left\langle u \nabla W \ast \rho_W, \nabla u \right\rangle _{L^{2}} -\left\langle \rho_W \nabla W \ast u, \nabla u \right\rangle _{L^{2}} \lesssim_{d} \left\| \nabla W \right\|_{L^{\infty}}\left\| \rho_W \right\|_{L^{2}} \left\| u \right\|_{L^{2}}\left\| \nabla u \right\|_{L^{2}}.
\] 
    Using the same arguments as \cite[Theorem 5]{mckv}, one shows that \(\sup_{0\leq t\leq T}\left\| \rho_W(t) \right\|_{L^{2}} \leq C \left( T,M, \left\| \phi  \right\|_{L^{2}} \right) \).
    Putting the two bounds together and using Young's product inequality to absorb the first order terms, we obtain
    \[
        \frac{\mathrm{d}}{\mathrm{d}t} \left\| u \right\|_{L^{2}}^{2} + \left\| \nabla u \right\|_{L^{2}}^{2} \leq C \left(\left\| u \right\|_{L^{2}}^{2} + \left\| f \right\|_{L^{2}}^{2}  \right) 
    \] with \(C=C \left( T,M, \left\| \phi  \right\|_{L^{2}} \right)<\infty \). From there, we directly conclude, using Gronwall's lemma and integrating from \(0\) to \(T\), that
\begin{equation}\label{eq:first-estimate}
    \sup_{0\leq t\leq T} \left\| u(t) \right\|_{L^{2}}^{2} + \left\| u \right\|_{L^{2}_{T}H^{1}}^{2} \leq C \left\| f \right\|_{L^{2}_{T}L^{2}}^{2}
\end{equation} with \(C=C \left( T,M, \left\| \phi  \right\|_{L^{2}} \right)<\infty \). To proceed, multiply \eqref{eq:linPDE} by \(\partial_{t}u \eqqcolon u'\) and integrate by parts again:
\begin{equation}\label{eq:second-estimate1}
    \left\| u' \right\|_{L^{2}}^{2} + \frac{1}{2} \frac{\mathrm{d}}{\mathrm{d}t} \left\| \nabla u \right\|^{2}_{L^{2}} = \left\langle \nabla \cdot \left(\rho_W \nabla W \ast u \right), u' \right\rangle _{L^{2}} + \left\langle \nabla \cdot \left(u \nabla W \ast \rho_W \right), u' \right\rangle _{L^{2}} + \left\langle f, u' \right\rangle _{L^{2}}.
\end{equation}
Again, the Cauchy-Schwarz inequality and Lemma \ref{lem:trilin-op-T-bound} equation \eqref{eq:trilin-leibniz} with \(k=1\) and \(l=0\) give us
\[
    \left\langle \nabla \cdot \left(\rho_W \nabla W \ast u \right), u' \right\rangle _{L^{2}} + \left\langle \nabla \cdot \left(u \nabla W \ast \rho_W \right), u' \right\rangle _{L^{2}} \lesssim_{d} \left\| \nabla W \right\|_{L^{\infty}}\left\| \rho_W \right\|_{H^{1}}\left\| u \right\|_{H^{1}} \left\| u' \right\|_{L^{2}}.
\]
By \cite[Theorem 5]{mckv}, we have that \(\sup_{0 \leq t \leq T}\left\| \rho_W(t) \right\|_{H^{1}} \leq C \left( d,T,M, \| \phi \|_{H^{1}} \right) \)
which, after packing the \(u'\) terms on the left-hand side using Young's product inequality, gives us
\[
    \left\| u' \right\|_{L^{2}}^{2} + \frac{1}{2} \frac{\mathrm{d}}{\mathrm{d}t} \left\| \nabla u \right\|^{2}_{L^{2}} \lesssim \left\| u \right\|_{H^{1}}^{2} + \left\| f \right\|_{L^{2}}^{2}.
\]
Integrating in time and using the first estimate \eqref{eq:first-estimate} as well as \(u(0)=0\) and the (time and space) regularity of Galerkin approximations, we obtain our second estimate
\begin{equation}\label{eq:second-estimate2}
    \left\| u' \right\|_{L^{2}_{T}L^{2}}^{2} + \sup_{0\leq t\leq T} \left\| u(t) \right\|_{H^{1}}^{2} \leq C \left\| f \right\|_{L^{2}_{T}L^{2}}^{2} 
\end{equation} with \(C=C \left( d,T,M, \| \phi \|_{H^{1}} \right)<\infty \). 

Regularity in \(L^{2}_{T}H^{2}\) now follows directly from the previous estimates. Multiply \eqref{eq:linPDE} by \(\Delta u\) and integrate
\[
\left\| \Delta u \right\|_{L^{2}}^{2} = \left\langle u',\Delta u \right\rangle_{L^2}  - \left\langle  \nabla \cdot (u \nabla W \ast \rho_W), \Delta u \right\rangle_{L^2} - \left\langle \nabla \cdot \left( \rho_W \nabla W \ast u \right), \Delta u \right\rangle_{L^2} - \left\langle f, \Delta u \right\rangle_{L^2}.
\]
Using Lemma \ref{lem:trilin-op-T-bound} equation \eqref{eq:trilin-leibniz} as before, as well as the Cauchy-Schwarz inequality and Young's product inequalities, we see 
\[
\left\| \Delta u \right\|_{L^{2}}^{2} \lesssim \left\| u' \right\|_{L^{2}}^{2} + \left\| u \right\|_{H^{1}}^{2} + \left\| f \right\|_{L^{2}}^{2}.
\]
Recall now the bounds on \(u\) and \(u'\) from \eqref{eq:first-estimate} and \eqref{eq:second-estimate2} and simply integrate in time to find that
\[
\left\| \Delta u \right\|_{L^{2}_{T}L^{2}}^{2} \leq C \left\| f \right\|_{L^{2}_{T}L^{2}}^{2}
\]
with \(C = C \left(d,T,M,\| \phi \|_{H^{d+1}} \right) \). Since we know that for \(u\) smooth enough, \(\left\| \Delta u \right\|_{L^{2}}= \left\| D^{2}u \right\|_{L^{2}}\) (using e.g. Plancherel's identity), we can combine the estimate above with \eqref{eq:first-estimate} to obtain regularity in \(L^{2}_{T}H^{2}\), which finishes the proof.
\end{proof}

\begin{theorem}[Higher space regularity]\label{thm:linPDE-higher-regularity}
    Suppose that the hypotheses of Theorem \ref{thm:existence-linPDE} hold, and consider the strong solution \(u\) of \eqref{eq:linPDE} given by that theorem. Then, if in addition \(f \in L^{2}([0,T];H^{k})\) for some \( 0 \leq k \leq \beta-1\),
    \[
    u\in L^{2}([0,T];H^{k+2}) \cap H^{1}([0,T];H^{k}).
    \]
    Moreover, there exists a constant \(C = C(d, k, T, M, \left\| \phi \right\|_{H^{k+1}})\) such that
    \[
        \left\| u \right\|_{L^{2}([0,T];H^{k+2})} + \left\| u \right\|_{L^{\infty}([0,T];H^{k+1})} + \left\| u' \right\|_{L^{2}([0,T];H^{k})} \leq C \left\| f \right\|_{L^{2}([0,T];H^{k})}
    \] where \(u' = \frac{\mathop{}\!\mathrm{d}u}{\mathop{}\!\mathrm{d}t}\). 
\end{theorem}

\begin{proof}
    Simply repeat the proof of Theorem \ref{thm:existence-linPDE} using the \(H^{k}\) inner-product instead of the \(L^{2}\) inner-product. Repeating the proof is possible since for \(k \leq \beta - 1\), \cite[Theorem 5]{mckv} guarantees that 
\[
    \sup_{0\leq t \leq T} \left\| \rho_W \right\|_{H^{k+1}} \leq C(d, k, T, M, \left\| \phi \right\|_{H^{k+1}}).
\]     

\end{proof}

\begin{theorem}[Estimates in weaker norms]\label{thm:linPDE-weak-norms} 
    Suppose that the hypotheses of Theorem \ref{thm:existence-linPDE} hold, and consider the strong solution \(u\) of \eqref{eq:linPDE} given by that theorem. Then for any \(1 \leq s < \beta + 1 - \frac{d}{2}\),
    \[
        \left\| u \right\|_{L^{2}([0,T];H^{-s+2})} + \left\| u \right\|_{L^{\infty}([0,T];H^{-s+1})} + \left\|u'\right\|_{L^{2}([0,T];H^{-s})} \leq C \left\| f \right\|_{L^{2}([0,T];H^{-s})}.
    \] 
    with \(C=C(s,d,T,M,\left\| \phi \right\|_{H^{k}})\), \(k=\min(\beta,\lfloor  s-1+d/2 \rfloor + 1)\). 
\end{theorem}

\begin{proof}
    Repeat the estimates in the proof of Theorem \ref{thm:existence-linPDE}, this time taking inner products in \(H^{-s}\). We sketch the main steps. By the Cauchy-Schwarz inequality, 
    \[
        -\left\langle \rho_W \nabla W \ast u, \nabla u \right\rangle _{H^{-s}} -\left\langle u \nabla W \ast \rho_W , \nabla u \right\rangle _{H^{-s}} \leq \left\| \nabla u \right\|_{H^{-s}} \left( \| \rho_W \nabla W \ast u \|_{H^{-s}} + \| u \nabla W \ast \rho_W \|_{H^{-s}} \right).
    \]
    To bound the terms inside the parentheses, we use the duality \(\| u \|_{H^{-s}} = \sup_{\| v \|_{H^{s}}\leq 1} \left\langle u,v \right\rangle _{L^{2}}\). Pick any \(v \in H^{s}\) and note, using the proof of Lemma \ref{lem:trilin-op-T-bound} equation \eqref{eq:trilin-leibniz} with \(k=s\) and \(l=1\), 
    \[
    \left\langle u \nabla W \ast \rho_W ,v \right\rangle _{L^{2}} = \left\langle u,v\nabla W \ast \rho_W  \right\rangle _{L^{2}} \leq \left\| u \right\|_{H^{-s}} \left\| v\nabla W \ast \rho_W \right\|_{H^{s}} \leq \left\| u \right\|_{H^{-s}} \left\| W \right\|_{W^{2,\infty}} \left\| \rho_W \right\|_{H^{s-1}} \left\| v \right\|_{H^{s}}.
    \]
    Similarly, letting \(\check{W}(x)\coloneqq W(-x)\) and using Fubini,
    \[
        \left\langle \rho_W \nabla W \ast u,v \right\rangle _{L^{2}} = \left\langle u,\nabla \check{W} \ast \left( \rho_W v \right)   \right\rangle _{L^{2}} \leq \left\| u \right\|_{H^{-s}} \left\| \nabla \check{W} \ast \left( \rho_W v \right)\right\|_{H^{s}} \leq \left\| u \right\|_{H^{-s}} \left\| W \right\|_{W^{2,\infty}} \left\| \rho_W \right\|_{C^{s-1}} \left\| v \right\|_{H^{s}}.
    \]
    Since, by assumption, \(s-1+ d/2 < \beta\), the Sobolev embeddings together with \cite[Theorem 5]{mckv} imply that 
    \[
    \sup_{0\leq t \leq T} \| \rho_W \|_{C^{s-1}} \leq C(s,d,T,M,\| \phi \|_{H^{k}}) < \infty,
    \] 
    with \(k=\min(\beta,\lfloor  s-1+d/2 \rfloor + 1)\). Therefore, taking the supremum over all \(v \in H^{s}\) with \(\| v \|_{H^{s}} \leq 1\), we obtain
    \[
    \| \rho_W \nabla W \ast u \|_{H^{-s}} + \| u \nabla W \ast \rho_W \|_{H^{-s}} \leq C \| u \|_{H^{-s}},
    \]with \(C=C(d,s,T,M,\| \phi \|_{H^{\beta}})\). Repeating these estimates for the other steps in the proof of Theorem \ref{thm:existence-linPDE} then yields the desired result.
\end{proof}

\begin{theorem}[Improved time regularity]\label{thm:linPDE-time-regularity} 
    Suppose that the hypotheses of Theorem \ref{thm:existence-linPDE} hold, and consider the strong solution \(u\) of \eqref{eq:linPDE} given by that theorem. If moreover \(f'\in L^{2}([0,T];L^{2})\) then,
\[
   u' \in L^{2}([0,T];H^{1}) \cap H^{1}([0,T];H^{-1}).
\]     
Moreover, there exists a constant \(C = C(d, T, M, \left\| \phi \right\|_{H^{1}})\) such that
\[
    \left\| \frac{\mathrm{d}u}{\mathrm{d}t} \right\|_{L^{2}([0,T];H^{1})} + \left\| \frac{\mathrm{d}u}{\mathrm{d}t} \right\|_{L^{\infty}([0,T];L^{2})} + \left\| \frac{\mathrm{d}^{2}u}{\mathrm{d}t^{2}} \right\|_{L^{2}([0,T];H^{-1})} \leq C \left\| f \right\|_{H^{1}([0,T];L^{2})}
\]
\end{theorem}
\begin{proof}
    We essentially repeat steps 4 to 7 of the proof of \cite[Sec. 7.1.3 Theorem 5]{evans2010partial}, using Lemma \ref{lem:trilin-op-T-bound} equation \eqref{eq:trilin-leibniz} to bound the nonlocal terms, as in Theorem \ref{thm:existence-linPDE}.
    
    Throughout the proof we denote by \(f'\) the time derivative of \(f\).
    We give the a priori estimate and refer to the proof of \cite[Sec. 7.1.3 Theorem 5]{evans2010partial} for the full Galerkin argument. First, we differentiate the equation in time and take the \(L^{2}\) inner-product with some \(u'\) to obtain, upon integration by parts,
\begin{align*}
    \frac{1}{2} \frac{\mathrm{d}}{\mathrm{d}t} \left\| u' \right\|_{L^{2}} + \left\| \nabla u' \right\|^{2}_{L^{2}} = &\left\langle f',u' \right\rangle\\ 
     &- \left\langle \rho_W' \nabla W \ast u, \nabla u' \right\rangle - \left\langle \rho_W \nabla W \ast u', \nabla u' \right\rangle \\
     &- \left\langle u' \nabla W \ast \rho_W, \nabla u' \right\rangle - \left\langle u \nabla W \ast \rho_W', \nabla u' \right\rangle.
\end{align*}
From \cite[Theorem 5]{mckv}, we know that \(\sup_{0 \leq t \leq T} \left\| \rho_W'(t) \right\|_{L^{2}} \leq C \left( d,T,M, \left\| \phi \right\|_{H^{1}} \right)  \) and  \(\sup_{0\leq t\leq T}\left\| \rho_W(t) \right\|_{L^{2}} \leq C \left( M, \left\| \phi  \right\|_{L^{2}},T \right) \). Combining these estimates with Lemma \ref{lem:trilin-op-T-bound} equation \eqref{eq:trilin-leibniz} with \(k=0\) and \(l=0\) gives
\[
    \frac{1}{2} \frac{\mathrm{d}}{\mathrm{d}t} \left\| u' \right\|_{L^{2}} + \left\| \nabla u' \right\|^{2}_{L^{2}} \lesssim \left\| f' \right\|_{L^{2}}\left\| u' \right\|_{L^{2}}+ \left\| u \right\|_{L^{2}} \left\| \nabla u' \right\|_{L^{2}} +  \left\| u' \right\|_{L^{2}} \left\| \nabla u' \right\|_{L^{2}}. 
\]
Recall from Theorem \ref{thm:existence-linPDE} that \(\sup_{0 \leq t \leq T}\left\| u \right\|_{L^{2}} \leq C \left( M, \left\| \phi  \right\|_{H^{1}}, d, T \right) \left\| f \right\|_{L^{2}_{T}L^{2}}\) and use Young's product inequality to obtain, after integrating in time,
\[
\sup_{0 \leq t \leq T} \left\| u' \right\|_{L^{2}}^{2} + \int_{0}^{T} \left\| u'(t) \right\|_{H^{1}}^{2} \mathop{}\!\mathrm{d}t \lesssim  \left\| f \right\|_{H^{1}_{T}L^{2}}^{2} + \left\| u'(0) \right\|_{L^{2}}^{2}.
\]
Now, using the PDE \eqref{eq:linPDE} we see that \(u'(0)=f(0)\), the latter being well-defined by virtue of the Sobolev embeddings (see e.g. \cite[Sec. 5.9.2. Theorem 2]{evans2010partial}). This implies that \(f\) is time continuous as an \(L^{2}\)-valued function, and \(\sup_{0 \leq t \leq T} \left\| f(t) \right\|_{L^{2}} \lesssim_{T}\left\| f \right\|_{H^{1}_{T}L^{2}}\). We finally obtain
\[
\sup_{0 \leq t \leq T} \left\| u'(t) \right\|_{L^{2}}^{2} + \int_{0}^{T} \left\| u'(t) \right\|_{H^{1}}^{2} \mathop{}\!\mathrm{d}t \leq C \left\| f \right\|_{H^{1}_{T}L^{2}}^{2}
\] with \(C = C \left( d,T,M, \left\| \phi \right\|_{H^{1}} \right)\). 

Now, we come back to the PDE, take time derivatives and \(L^{2}\) inner-products with some \(v \in H^{1}\):  
\begin{align*}
    \left\langle u'',v \right\rangle = \left\langle f',v \right\rangle &- \left\langle \nabla u', \nabla v \right\rangle \\
    & - \left\langle \rho_W' \nabla W \ast u, \nabla v \right\rangle - \left\langle \rho_W \nabla W \ast u', \nabla v \right\rangle - \left\langle u' \nabla W \ast \rho_W, \nabla v \right\rangle - \left\langle u \nabla W \ast \rho_W', \nabla v \right\rangle.
\end{align*}
Using once again the estimates from \cite[Theorem 5]{mckv} as well as Lemma \ref{lem:trilin-op-T-bound} equation \eqref{eq:trilin-leibniz} and the Cauchy-Schwarz inequality, we see that
\[
    \left\langle u'', \frac{v}{\left\| v \right\|_{H^{1}}} \right\rangle \lesssim \left\| u' \right\|_{H^{1}}  + \left\| u \right\|_{L^{2}} + \left\| f' \right\|_{L^{2}}.
\]
Therefore, by the definition of \(H^{-1}\) as a dual space,
\[
    \left\| u'' \right\|_{H^{-1}} \lesssim \left\| u' \right\|_{H^{1}}  + \left\| u \right\|_{L^{2}} + \left\| f' \right\|_{L^{2}}.
\]
Integrating in time and using the previous estimates yields the desired result.
\end{proof}

\section{Calculus in Banach spaces}

In this section, we collect some results from calculus in Banach spaces. 

\begin{theorem}\label{thm:abstract-derivative}
    Let \(E,F\) and \(G\) be Banach spaces and \(U\subset E\times F\) an open subset of \(E\times F\). Let \(f:U \subset E\times F \to G\) be \(C^{k}\) in the Fréchet sense for \(k \geq 1\). Let \(V\subset E\) an open subset of \(E\) and assume that there exists a function \(u:V \to F\) such that for all \(x \in V\),
    \begin{enumerate}
        \item \((x,u(x)) \in U\),
        \item \(f(x, u(x)) = 0\),
        \item \(D_2f(x,u(x)):F\to G\) is a linear homeomorphism.
    \end{enumerate}
    Additionally, assume that either \(u\) is continuous, or that for all \(x \in V\), the only solution \((x,v)\in U\) to the equation \(f(x,v)=0\) is \((x,u(x))\). Then, \(u\) is \(C^{k}\) in the Fréchet sense and its derivative \(D u(x) \in \mathcal{L}(E,F)\) is given by
    \begin{equation}\label{eq:ift-first-derivative}
        D u(x) [h] = - (D_2f(x,u(x)))^{-1} D_1 f(x,u(x)) [h]
    \end{equation} for any \(x \in V\) and \(h\in E\).
    
    For \(k \geq 2\), the second derivative \(D^2u(x) \in \mathcal{L}(E\times E, F)\) is given by
    \begin{equation}\label{eq:ift-second-derivative}
    \begin{aligned}
        D^2u(x) [h_1,h_2] = - (D_2f(x,u(x)))^{-1} \Big( D^{2}_{1,1}f(x,u(x)) &[ h_1,h_2 ] \\
        + D^{2}_{1,2}f(x,u(x)) &[ h_1,Du(x) [ h_2 ]] \\
        + D^{2}_{2,1}f(x,u(x)) &\cdot \left( Du(x) [ h_1 ], h_2 \right) \\
        + D^{2}_{2,2}f(x,u(x)) &\cdot \left( Du(x) [ h_1 ], Du(x) [ h_2 ] \right) \Big)
    \end{aligned}
    \end{equation} for any \(x \in V\) and \(h_1, h_2\in E\).

    For larger \(k\geq2\), the general formula is 
    \begin{equation}\label{eq:Faa-di-Bruno}
            D^ku(x) [h_1,\ldots,h_k] = - (D_2f(x,u(x)))^{-1} \sum_{\substack{\pi \in \Pi_k \backslash\{\pi_0\} \\ \pi \eqqcolon \left\{ P_1,\ldots, P_m \right\}  }} D^{\left| \pi \right|}f(x,u(x)) [ H_{P_{1}}, \ldots, H_{P_m} ] 
    \end{equation} where \(\Pi_k\) is the set of all partitions of \(\left\{ 1,\ldots,k \right\} \), \(\pi_0 \coloneqq
     \left\{ \left\{ 1,\ldots,k \right\} \right\} \) and
     \[
     H_{P_j} \coloneqq 
     \begin{cases}
        \left(h_i, D u(x) [h_i]\right) & \text { if }\left|P_j\right|=1 \text { and } P_j=\{i\}, \\ 
        \left(0, D^{\left|P_j\right|} u(x) [\left(h_i\right)_{i \in P_j}]\right) & \text { if }\left|P_j\right| \geq 2.
    \end{cases}
     \] 
\end{theorem}

\begin{proof}[Proof of Theorem \ref{thm:abstract-derivative}]
    First, we show that \(u\) is \(C^{1}\). For this, choose an arbitrary \(x_0\in V\). We start by applying the Banach space implicit function theorem, see for example \cite[Theorem (10.2.1)]{dieudonne}. We obtain that there exist a connected open neighbourhood \(U_{x_0}\) of \(x_0\) contained in \(V\) and a unique continuous function \(\tilde{u}:U_{x_0} \to F\) such that for all \(x \in U_{x_0}\), \((x,u(x)) \in U\) and \(f(x,\tilde{u}(x))=0\). Moreover, \(\tilde{u}\) is \(C^{1}\) in \(U_{x_0}\).

    Now, if \(u\) is continuous, then \(u=\tilde{u}\) in \(U_{x_0}\) by the uniqueness part of the implicit function theorem. This implies that \(u\) is \(C^{1}\) in \(U_{x_0}\) and, in particular, \(u\) is differentiable at \(x_0\). Since \(x_0\) was arbitrary, we obtain that \(u\) is differentiable in \(V\). Then, since \(u\) is \(C^{1}\) in \(U_{x_0}\) in particular its derivative is continuous at \(x_0\). Again, \(x_0\) was arbitray so \(Du\) is continuous in \(V\). If instead we assume that \(u(x)\) is the only solution to \(f(x,u(x))=0\) for every \(x\in V\), then we must have that \(u=\tilde{u}\) in \(U_{x_0}\) as well and \(u\) is again \(C^1\) by the preceding reasoning.     

    Now that we know that \(u\) is \(C^{1}\), the first formula \eqref{eq:ift-first-derivative} is directly obtained by differentiating the equation \(f(x,u(x))=0\) with respect to \(x\) using the chain rule. We can then use \eqref{eq:ift-first-derivative} to obtain that \(u\) is \(C^{k}\) by a simple bootstrap argument. Indeed, assume that \(u\) is \(C^{l}\) for some \(l <k\). Since \(f\) is \(C^{k}\), the chain rule gives us that \((D_2f(x,u(x)))^{-1} D_1 f(x,u(x))\) is \(C^{l}\). But then this means that \(Du\) itself is \(C^{l}\), i.e. \(u\) is \(C^{l+1}\). This proves the claim by induction. Once we know that \(u\) is \(C^{k}\), formula \eqref{eq:Faa-di-Bruno} (and therefore formula \eqref{eq:ift-second-derivative} as well) follows from using Faà di Bruno's multivariate formula on \(D^{k} (f \circ g) (x) = 0\) where we set \(g(x)=(x,u(x))\), and isolating the term involving the paritition \(\pi_0\).
\end{proof}

\begin{lemma}\label{lem:mult-is-smooth}
    Let \(f: E_1 \times \dots \times E_m \to G\) be a continuous multilinear map, i.e. \(f\) is \(m\)-linear and continuous. Then, \(f\) is \(C^{\infty}\) in the Fréchet sense.
\end{lemma}

\begin{proof}[Proof of Lemma \ref{lem:mult-is-smooth}]
    See (8.1.4) in \cite[Chapter VIII]{dieudonne}, and proceed by induction. 
\end{proof}

\subsection{Remaining proofs}\label{appendix:further-technical-results}

\begin{proof}[Proof of Proposition \ref{prop:derivatives-f}]
    We first prove that \(f\) is \(C^{\infty}\). A continuous multilinear map is \(C^{\infty}\) in the Fréchet sense, see Lemma \ref{lem:mult-is-smooth}. Consider the auxiliary maps \(g:F \times E \times F \to G\) and \(h: F \to G\)  defined by
    \[
    h(\rho) \coloneq \left( \partial_t \rho  - \Delta \rho, \rho(0) \right), \quad g(\rho,W,u) \coloneqq \left( - \nabla \cdot \left( \rho  \nabla W \ast u  \right) , 0  \right) = \left(- \mathcal{T}(\rho,W,u),0 \right).
    \] 
    By definition, \(h\) is linear. We show that it is continuous, i.e. bounded:
    \[
    \left\| h(\rho ) \right\|_{G} \leq \left\| \left( \partial_t - \Delta \right) \rho \right\|_{L^{2}_{T}H^{\beta-1}} + \left\| \rho(0) \right\|_{H^{\beta}} \leq \left\| \rho \right\|_{H^{1}_{T}H^{\beta-1}} + \left\| \rho \right\|_{L^{2}_{T}H^{\beta+1}} + C \big( \left\| \rho \right\|_{H^{1}_{T}H^{\beta-1}} + \left\| \rho \right\|_{L^{2}_{T}H^{\beta+1}} \big)  
    \] where we used \eqref{eq:trace-theorem-time}. This proves that \(\left\| h(\rho) \right\|_{G} \lesssim \left\| \rho \right\|_{F} \) and therefore \(h\) is $C^\infty$. Then, \(g\) is trilinear by definition and continuous by Lemma \ref{lem:trilin-bounded}, hence also $C^\infty$. Finally, \(f\) decomposes as
    \begin{equation}\label{eq:f-decomposition}
        f(W,\rho) = h(\rho) + g(\rho,W,\rho) + (0,-\phi)
    \end{equation} 
    which shows that \(f\) is smooth as a sum of smooth maps.

    The expressions for the derivatives of \(f\) follow as well from the decomposition \eqref{eq:f-decomposition}. From the chain rule the first derivative is 
    \[
    Df(W,\rho) [H,s] = Dh (\rho)[s] + Dg(\rho,W,\rho) [s,H,s] + 0 
    \]
    because the maps \((W,\rho) \mapsto \rho\) and \((W,\rho) \mapsto (\rho,W,\rho)\) are continuous linear and hence their own derivatives. Now, the rules of differentiation for continuous multilinear maps give:
    \[
    Df(W,\rho)[H,s] = h(s) + g(s,W,\rho) + g(\rho,H,\rho) + g(\rho,W,s)
    \] and hence 
    \begin{equation}\label{eq:first-derivatives}
        D_1f (W,\rho) [H] = g(\rho,H,\rho), \quad D_2f(W,\rho) [s] = h(s) + g(s,W,\rho) + g(\rho,W,s)
    \end{equation} which is exactly the expression given in the proposition. 

    For the second derivatives we differentiate the expressions in \eqref{eq:first-derivatives} with respect to $W$ and $\rho$. Since $D_1f$ does not depend on $W$ we see that $D_{1,1}^2f =0$, and similar to above 
    \[D_{2,1}^{2}f(W,\rho) [s,H] = g(\rho,H,s) + g(s,H,\rho) = D_{1,2}^{2}f(W,\rho) [H,s]
    \]
    Also the second equation in \eqref{eq:first-derivatives} is already linear in $\rho$ so we obtain 
       \[
    D_{2,2}^{2}f(W,\rho) [s_1,s_2] = g(s_1,W,s_2) + g(s_2,W,s_1),
    \] 
   as required.
\end{proof}

\begin{proof}[Proof of Lemma \ref{lem:bound-on-derivatives-of-forward-map}]  Recall the definition \(\mathcal{T}(r,V,s) \coloneqq \nabla \cdot \left( r \nabla V \ast s \right)\) from \eqref{eq:def-Tcal}.
    Throughout the proof, we will use the notation \(\left\| \ \cdot \  \right\|_{\gamma} \coloneqq  \left\| \ \cdot \  \right\|_{L^{2}([0,T];H^{\gamma+1})\cap H^{1}([0,T],H^{\gamma-1})}\) 
    We proceed by induction. The case \(k=1\) follows directly from the regularity estimate in Theorem \ref{thm:linPDE-higher-regularity} combined with Lemma \ref{lem:trilin-op-T-bound} equation \eqref{eq:trilin-algebra} with \(\nu=0\) and \eqref{eq:trace-theorem-time}:
    \[
    \left\| D \rho_W [h_1] \right\|_{\gamma} \lesssim \left\| \mathcal{T}(\rho_W,h_1,\rho_W) \right\|_{L^{2}([0,T],H^{\gamma-1})} \lesssim \left\| h_1 \right\|_{L^{2}} \left\| \rho_W \right\|_{\gamma}^{2}.
    \]
    Since \(\left\| \rho_W \right\|_{\gamma} \leq C(d,\gamma,T,M,\left\| \phi \right\|_{H^{\gamma}})\), this proves the base case.
    
    Now, let \(k \geq 2\) and assume the lemma holds for all \(l < k\). Formula \eqref{eq:Faa-di-Bruno} allows us to express \(D^{k}\rho_W\) in terms of the lower order terms. Namely, \(D^{k}\rho_W  [H_1, \dots, H_k]\) solves \eqref{eq:linPDE} with right-hand side forcing term \(F\) and initial condition \(\psi\) given by
    \[
    (F,\psi) \coloneqq - \sum_{\substack{\pi \in \Pi_k \backslash\{\pi_0\} \\ \pi \eqqcolon \left\{ P_1,\ldots, P_m \right\}  }} D^{\left| \pi \right|}f(W,\rho_W) [H_{P_{1}}, \ldots, H_{P_m}]
    \] 
    where we recall that \(\Pi_k\) is the set of all partitions of \(\left\{ 1,\ldots,k \right\} \), \(\pi_0 \coloneqq
     \left\{ \left\{ 1,\ldots,k \right\} \right\} \) and
     \[
     H_{P_j} \coloneqq 
     \begin{cases}
        \left(h_i, D \rho_W [h_i] \right) & \text { if }\left|P_j\right|=1 \text { and } P_j=\{i\}, \\ 
        \left(0, D^{\left|P_j\right|} \rho_W \cdot [h_i]_{i \in P_j}\right) & \text { if }\left|P_j\right| \geq 2 .
    \end{cases}
     \] 
    In particular, we note that the terms in the sum involve derivatives of \(f\) of order at least \(2\) since \(\pi \neq \pi_0\).  
    Proposition \ref{prop:derivatives-f} then shows that \(D^{2}_{1,2}f\) is constant in \(W\) and linear in \(\rho\), and \(D^{2}_{2,2}f\) is constant in \(\rho\) and linear in \(W\). Therefore, we see that \(D^{4}f = 0, D^{3}_{1,1,1}f = D^{3}_{1,1,2}f = D^{3}_{2,2,2}f = 0\), and 
    \[
    D^{3}_{1,2,2}f(W,\rho) [h,s_1,s_2] = - \left( \mathcal{T}(s_1,h,s_2) + \mathcal{T}(s_2,h,s_1), 0 \right).
    \] In other words, \(\psi =0\) and \(F\) is the sum of all the terms 
    \[
    \mathcal{T}(D^{\left| P_2 \right| }\rho_W [h_i]_{i \in P_2},h_j,D^{\left| P_3 \right| }\rho_W [h_i]_{i \in P_3}) + \mathcal{T}(D^{\left| P_3 \right| }\rho_W \cdot [h_i]_{i \in P_3},h_j,D^{\left| P_2 \right| }\rho_W \cdot [h_i]_{i \in P_2})
    \] for all partitions of size \(3\) and all \(j=1,\ldots ,k\) such that \(P_1=\{j\}\) (terms coming from \(D^{3}_{1,2,2}f\)),
    \[
    \mathcal{T}\left( D^{k-1}\rho_W [h_i]_{i \neq j}, h_j, \rho_W  \right) + \mathcal{T}\left(\rho_W, h_j, D^{k-1}\rho_W [h_i]_{i \neq j} \right) 
    \] for all \(j=1,\ldots ,k\) (terms coming from \(D^{2}_{1,2}f\)), and,
    \[
    \mathcal{T}(D^{\left| P_1 \right| }\rho_W [h_i]_{i \in P_1},W,D^{\left| P_2 \right| }\rho_W [h_i]_{i \in P_2}) + \mathcal{T}(D^{\left| P_2 \right| }\rho_W [h_i]_{i \in P_2},W,D^{\left| P_1 \right| }\rho_W [h_i]_{i \in P_1})
    \] for all partitions of size \(2\) (terms coming from \(D^{2}_{2,2}f\)).
    
    We can now use Theorem \ref{thm:linPDE-higher-regularity} to obtain that
    \[
    \left\| D^{k}\rho_W [h_1, \dots, h_k] \right\|_{L^{2}([0,T];H^{\gamma+1})\cap H^{1}([0,T],H^{\gamma-1})} \leq C \left\| F \right\|_{L^{2}([0,T],H^{\gamma-1})}
    \] with \(C=C(d,\gamma,T,M, \left\| \phi \right\|_{H^{\gamma }})\). To finish the induction step and conclude the proof, we use the triangle inequality to split \(F\) into all of its terms. We control each term using Lemma \ref{lem:trilin-op-T-bound} equation \eqref{eq:trilin-algebra} with \(\nu=0\), combined with \eqref{eq:trace-theorem-time}. 

    Terms coming from \(D^{3}_{1,2,2}f\): 
\begin{align*}
&\left\| \mathcal{T}\bigl(D^{|P_2|}\rho_W [h_i]_{i \in P_2},\, h_j,\,
   D^{|P_3|}\rho_W [h_i]_{i \in P_3}\bigr) \right\|_{L^{2}([0,T],H^{\gamma-1})} \\
&\hspace{9em} \lesssim \left\|h_j\right\|_{L^{2}} 
   \sup_{0\leq t \leq T}\left\|D^{|P_2|}\rho_W [h_i]_{i \in P_2}\right\|_{H^{\gamma}}
   \left\|D^{|P_3|}\rho_W [h_i]_{i \in P_3}\right\|_{L^{2}([0,T];H^{\gamma+1})} \\
&\hspace{9em} \lesssim \left\|h_j\right\|_{L^{2}}
   \left\|D^{|P_2|}\rho_W [h_i]_{i \in P_2}\right\|_{\gamma}
   \left\|D^{|P_3|}\rho_W [h_i]_{i \in P_3}\right\|_{\gamma} \\
&\hspace{9em} \leq C \left\| h_1 \right\|_{L^{2}} \ldots \left\| h_k \right\|_{L^{2}}
\end{align*} where we used the induction hypothesis in the last inequality. The other terms are bounded analogously. 

Terms coming from \(D^{2}_{1,2}f\):
\begin{align*}
&\left\| \mathcal{T}\left( D^{k-1}\rho_W [h_i]_{i \neq j}, h_j, \rho_W  \right) \right\|_{L^{2}([0,T],H^{\gamma-1})}  \lesssim \left\|h_j\right\|_{L^{2}} 
   \sup_{0\leq t \leq T}\left\|D^{k-1}\rho_W [h_i]_{i \neq j}\right\|_{H^{\gamma}}
   \left\|\rho_W\right\|_{L^{2}([0,T];H^{\gamma+1})} \\
&\hspace{6em} \lesssim \left\|h_j\right\|_{L^{2}} 
   \left\|D^{k-1}\rho_W [h_i]_{i \neq j}\right\|_{\gamma} \leq C \left\| h_1 \right\|_{L^{2}} \ldots \left\| h_k \right\|_{L^{2}}
\end{align*} where we used \cite[Theorem 5]{mckv} in the second inequality, and the induction hypothesis in the last one. The other terms are bounded analogously.

Terms coming from \(D^{2}_{2,2}f\):
\begin{align*}
&\left\| \mathcal{T}(D^{\left| P_1 \right| }\rho_W [h_i]_{i \in P_1},W,D^{\left| P_2 \right| }\rho_W [h_i]_{i \in P_2}) \right\|_{L^{2}([0,T],H^{\gamma-1})} \\
&\hspace{9em} \lesssim \left\|W\right\|_{L^{2}} 
   \sup_{0\leq t \leq T}\left\|D^{\left| P_1 \right| }\rho_W [h_i]_{i \in P_1}\right\|_{H^{\gamma}}
   \left\|D^{\left| P_2 \right| }\rho_W [h_i]_{i \in P_2}\right\|_{L^{2}([0,T];H^{\gamma+1})} \\
&\hspace{9em} \lesssim \left\|W\right\|_{W^{2,\infty}} 
   \left\|D^{\left| P_1 \right| }\rho_W [h_i]_{i \in P_1}\right\|_{\gamma}
   \left\|D^{\left| P_2 \right| }\rho_W [h_i]_{i \in P_2}\right\|_{\gamma} \\
&\hspace{9em} \leq C \left\| h_1 \right\|_{L^{2}} \ldots \left\| h_k \right\|_{L^{2}}
\end{align*} where we used the boundedness of the domain to write \(\left\| W \right\|_{L^{2}} \leq \left\| W \right\|_{W^{2,\infty}}\), and the induction hypothesis in the last inequality. The other terms are bounded analogously.

\end{proof}

\textbf{Acknowledgements.} The authors  acknowledge funding from an ERC Advanced Grant (UKRI G116786) as well as by EPSRC programme grant EP/V026259. RN would like to thank Jan Bohr for pointing out the reference \cite{AG07}, and both authors thank Greg Pavliotis for multiple discussions.

\bibliographystyle{plain}
\bibliography{refs}

\textsc{Department of Pure Mathematics \& Mathematical Statistics}

\textsc{University of Cambridge}, Cambridge, UK

Email: nickl@maths.cam.ac.uk, afc51@cam.ac.uk

\end{document}